\documentclass[11pt]{article}

\usepackage{amsfonts,amsmath,amssymb,amsthm}

\newtheorem{thm}{Theorem}[section]
\newtheorem{prop}[thm]{Proposition}
\newtheorem{lem}[thm]{Lemma}
\newtheorem{cor}[thm]{Corollary}

\newtheorem{asm}{Assumption}

\theoremstyle{remark}

\theoremstyle{definition}

\newcommand{\ra}{\rightarrow}

\newcommand{\Ra}{\Rightarrow}

\newcommand{\N}{\mathbb N}     
\newcommand{\Q}{\mathbb Q}     
\newcommand{\R}{\mathbb R}     
\newcommand{\Z}{\mathbb Z}     

\renewcommand{\a}{\alpha}
\renewcommand{\b}{\beta}

\renewcommand{\d}{\delta}

\newcommand{\e}{\varepsilon}

\newcommand{\w}{\omega}

\renewcommand{\l}{\lambda}

\newcommand{\s}{\sigma}

\renewcommand{\P}{\mathbb{P}}   
\newcommand{\E}{\mathbb{E}}   

\newcommand{\bigo}{\mathcal{O}}
\newcommand{\bw}{\bar{\w}}  

\author{Jonathon Peterson\thanks{
School of Mathematics, University of Minnesota,
206 Church St. SE, Minneapolis, MN 55455.} \thanks{
The research of the author was
partially supported
by NSF grant  DMS-0503775 and by a Doctoral Dissertation Fellowship from the University of Minnesota.}
}
\title{Quenched Limits 
for Transient, Ballistic, Sub-Gaussian 
One-Dimensional Random Walk in Random Environment}

\date{May 12, 2008}

\begin{document}
\maketitle
\begin{abstract}
We consider a nearest-neighbor,
one-dimensional random walk $\{X_n\}_{n\geq 0}$ in a random
i.i.d. environment,
in the regime where the walk is transient with speed $v_P > 0$ and there exists an $s\in(1,2)$ such that the annealed law  of $n^{-1/s} ( X_n - n v_P )$ converges to a stable law of parameter $s$.
Under the quenched law (i.e., conditioned on the environment), we show that
no limit laws are possible. In particular we show that there exist sequences $\{t_k\}$ and
$\{t_k'\}$ depending on
the environment only, such that a quenched central limit theorem holds along the subsequence $t_k$, but the quenched limiting distribution along the subsequence $t_k'$ is a centered reverse exponential distribution. This complements the results of a recent paper of Peterson and Zeitouni (\emph{arXiv:math/0704.1778v1 [math.PR]}
) which handled the case when the parameter $s\in(0,1)$. 
\end{abstract}

\begin{abstract}
On examine des marches al\'eatoires
unidimensionnelles 
en milieu al\'eatoire avec 
un environnement i.i.d.,
dans le r\'egime o\`u la marche est transiente avec vitesse $v_P > 0$ 
et o\`u il existe $s\in(1,2)$ tel que la 
loi ``annealed'' (i.e., moyenn\'{e}e) de $n^{-1/s} ( X_n - n v_P )$ converge vers une loi stable de param\`etre $s$.
Sous la loi ``quenched'' (i.e. conditionnelement \`a l'environnement) on montre qu'il n'existe pas de loi limite. En particulier on prouve qu'il existe des suites $\{t_k\}$ et $\{t_k'\}$, d\'ependant de l'environnement, tel qu'un th\'eor\`eme de limite centrale quenched est valide le long de la suite $t_k$, mais o\`u
la distribution limite suivant la suite $t_k'$ est une 
distribution centr\'ee exponentielle inverse.
Ceci compl\`ete les r\'esultats 
d'un article r\'ecent de 
Peterson et Zeitouni (\emph{arXiv:math/0704.1778v1 [math.PR]}
) qui traitait 
le case de param\`{e}tre
$s\in(0,1)$. 
\end{abstract}

\vspace{0.5cm}
\noindent{\sc Key Words:}
 Random walk,  random environment.\\
{AMS (1991) subject classifications:} Primary 60K37, Secondary 60F05, 82C41,
82D30\,.

\begin{section}{Introduction, Notation, and Statement of Main Results}
Let $\Omega = [0,1]^\Z$, and let $\mathcal{F}$ be the Borel $\s-$algebra on $\Omega$. A random environment is an $\Omega$-valued random variable $\w = \{\w_i\}_{i\in\Z}$ with distribution $P$. In this paper we will assume that $P$ is a product measure on $\Omega$.
The \emph{quenched} law $P_\w^x$ for a random walk $X_n$ in the environment $\w$ is defined by
\[
P_\w^x( X_0 = x ) = 1, \quad \text{and} \quad
P_\w^x\left( X_{n+1} = j | X_n = i \right) =
\begin{cases}
\w_i &\quad \text{if } j=i+1, \\
1-\w_i &\quad \text{if } j=i-1.
\end{cases}
\]
$\Z^\N$ is the space for the paths of the random walk $\{X_n\}_{n\in\N}$,
and let $\mathcal{G}$ denote the $\s-$algebra generated by the cylinder sets.
Note that for each $\w \in \Omega$, $P_\w$ is a probability measure
on $( \Z^\N, \mathcal{G} )$, and for each $G\in \mathcal{G}$,
$P_\w^x(G):(\Omega, \mathcal{F}) \ra [0,1]$ is a measurable
function of $\w$.  Expectations under the law $P_\w^x$ are denoted $E_\w^x$.
The \emph{annealed} law for the random walk in random
environment $X_n$ is defined by
\[
\P^x(F\times G) = \int_F P_\w^x(G)P(d\w),
\quad F\in \mathcal{F},  G\in \mathcal{G}\!.
\]
For ease of notation we will use $P_\w$ and $\P$ in place
of $P_\w^0$ and $\P^0$ respectively. We will also use $\P^x$ to
refer to the marginal on the space of paths, i.e. $\P^x(G)=
\P^x(\Omega\times G) = E_P\left[ P^x_\w(G) \right]$ for
$G\in \mathcal{G}$. Expectations under the law $\P$ will be written $\E$.

A simple criterion for recurrence of a one-dimensional RWRE and a formula for the speed of
transience was given by Solomon in \cite{sRWRE}. For any integers
$i\leq j$ define
\begin{equation}
\rho_i := \frac{1-\w_i}{\w_i}, \quad \text{and}\quad
\Pi_{i,j} := \prod_{k=i}^j \rho_k\,. \label{rhodef}
\end{equation}
Then, $X_n$ is transient to the right (resp. to the left)
if $E_P(\log \rho_0) < 0$, (resp. $E_P \log \rho_0 > 0$) and recurrent
if $E_P (\log \rho_0) = 0$ (henceforth we will write $\rho$ instead of
$\rho_0$ in expectations involving only $\rho_0$). In the case where
$E_P \log\rho < 0$ (transience to the right),
Solomon established the following law of large numbers
\begin{equation}
v_P:= \lim_{n\ra\infty} \frac{X_n}{n} =
\lim_{n\ra\infty} \frac{n}{T_n} = \frac{1}{\E T_1}, \quad \P-a.s. \label{XTLLN}
\end{equation}
where $T_n:= \min\{k \geq 0:X_k=n\}$. 
For any integers $i<j$ define
\begin{equation}
W_{i,j} := \sum_{k=i}^j \Pi_{k,j}, \quad \text{and}
\quad W_j := \sum_{k\leq j} \Pi_{k,j}\,. \label{Wdef}
\end{equation}
When $E_P \log \rho< 0$, it was shown in 
\cite{zRWRE}
that
\begin{equation}
E_\w^j T_{j+1} = 1+2W_j < \infty, \quad P-a.s., \label{QET}
\end{equation}
and thus $v_P =
1/(1+2E_P W_0)$. Since $P$ is a product measure, $E_P W_0 =
\sum_{k=1}^\infty \left(E_P \rho\right)^k$. In particular, $v_P > 0$ if
$E_P \rho < 1$.

Kesten, Kozlov, and Spitzer \cite{kksStable} determined the
annealed limiting distribution of a RWRE with $E_P \log \rho < 0$, i.e.
transient to the right. They derived the limiting distributions for the walk by first establishing a stable
limit law of index $s$ for $T_n$, where $s$ is defined by the equation
$
E_P\rho^s = 1$. 
In particular, they showed that when $s\in(1,2)$ there exists a  $b>0$ such that
\begin{equation}
\lim_{n\ra\infty} \P\left( \frac{T_n-\E T_n}{n^{1/s}} \leq x \right) =
L_{s,b}(x)\, \label{annealedstableT}
\end{equation}
and
\begin{equation}
\lim_{n\ra\infty} \P\left( \frac{X_n - n v_P}{v_P^{1+1/s} n^{1/s}} \leq x
\right) = 1-L_{s,b}(-x), \label{annealedstableX}
\end{equation}
where $L_{s,b}$ is the distribution function for a stable random
variable with characteristic function
\[
\hat{L}_{s,b}(t)= \exp\left\{ -b|t|^s \left(
1-i\frac{t}{|t|}\tan(\pi s/2)  \right) \right\}. 
\]
While the annealed limiting distributions for transient
one-dimensional RWRE have been known for quite a while, the corresponding
quenched limiting distributions have remained largely unstudied
until recently. In the case when $s>2$, Goldsheid \cite{gQCLT} and Peterson \cite{pThesis}
independently proved that a quenched CLT holds with a
random (depending on the environment)
centering. Previously, in \cite{kmCLT} and \cite{zRWRE}
it had only been shown that the limiting statements for
the quenched CLT with random centering held in probability (rather than almost surely). In the case when $s<1$ it was shown in \cite{pzSL1} that no quenched limiting distribution exists for the RWRE. In particular, it was shown that $P-a.s.$ there exist two different random sequences $t_k$ and $t_k'$ such that the behavior of the RWRE is either localized (concentrated in a interval of size $\log^2 t_k'$) or spread out (scaling of order $t_k^{s}$). 

In this paper, we analyze the quenched limiting distributions of a
one-dimensional transient RWRE in the case $s\in(1,2)$. We show that, as in the case when $s<1$, there is no quenched limiting distribution of the random walk. However, as shown in Section \ref{gp}, the existence of a positive speed for the random walk allows us to transfer limiting distributions from $T_n$ to $X_n$. 
Throughout the paper, we will
make the following assumptions:
\begin{asm} \label{essentialasm}
$P$ is a product measure on $\Omega$ such that
\begin{equation}
E_P \log\rho < 0 \quad\text{and}\quad E_P \rho^s = 1 \text{ for
some } s>0 . \label{zerospeedregime}
\end{equation}
\end{asm}
\begin{asm}
The distribution of $\log \rho$ is non-lattice under
$P$ and $E_P ( \rho^s \log\rho ) < \infty$.  \label{techasm}
\end{asm}
\noindent\textbf{Remarks:}\\
\textbf{1.} Assumption \ref{essentialasm}
contains the essential assumptions for our results. The technical conditions contained in Assumption \ref{techasm} were also invoked in \cite{kksStable} and \cite{pzSL1}. \\
\textbf{2.} Since $E_P \rho^\gamma$ is a convex function of
$\gamma$, the two statements in \eqref{zerospeedregime} give that
$E_P \rho^\gamma < 1$ for all
$0<\gamma<s$ and $E_P \rho^\gamma > 1$ for all $\gamma > s$. In particular this implies that $v_P > 0 \iff s > 1$. The main results of this paper are for $s\in(1,2)$, but many statements hold for a wider range of $s$. If no mention is made of bounds on $s$ then it is assumed that the statement holds for all $s>0$.\\
\textbf{3.} The cases $s\in\{1,2\}$ are not covered by \cite{pzSL1} or by this paper. It is not clear whether or not a quenched CLT holds in the case $s=2$, but we suspect that the results for $s=1$ will be similar to those of the cases $s\in(0,1)$ and $s\in(1,2)$  - i.e. no quenched limiting distribution for the random walk. However, since $s=1$ is the bordering case between the zero-speed and positive-speed regimes the analysis is likely to be more technical (as was also the case in \cite{kksStable}). 

Let $\Phi(x)$ and $\Psi(x)$ be the distribution functions for a gaussian and exponential random variable respectively. That is, 
\[
\Phi(x):= \int_{-\infty}^x \frac{1}{\sqrt{2\pi}} e^{-t^2/2} dt \quad\text{and}\quad \Psi(x):= \begin{cases} 0 & x < 0 \\ 1-e^{-x} & x\geq 0 \end{cases} \,.
\]
Our main results are the following:
\begin{thm}\label{qCLT}
Let Assumptions \ref{essentialasm} and \ref{techasm} hold, and let $s\in(1,2)$. Then $P-a.s.$ there exists a random subsequence $n_{k_m}=n_{k_m}(\w)$ of $n_k=2^{2^k}$ and non-deterministic random variables $v_{k_m,\w}$ such that 
\[
\lim_{m\ra\infty} P_\w\left( \frac{ T_{n_{k_m}} - E_\w T_{n_{k_m}} }{ \sqrt{v_{k_m,\w}} } \leq x \right) = \Phi(x), \qquad \forall x\in\R,
\]
and 
\[
\lim_{m\ra\infty} P_\w\left( \frac{ X_{t_m} - n_{k_m} }{v_P \sqrt{v_{k_m,\w}} } \leq x \right) = \Phi(x), \qquad \forall x\in\R,
\]
where $t_m=t_m(\w):= \left\lfloor E_\w T_{n_{k_m}} \right\rfloor$.
\end{thm}
\begin{thm}\label{qEXP}
Let Assumptions \ref{essentialasm} and \ref{techasm} hold, and let $s\in(1,2)$. Then $P-a.s.$ there exists a random subsequence $n_{k_m}=n_{k_m}(\w)$ of $n_k=2^{2^k}$ and non-deterministic random variables $v_{k_m,\w}$ such that 
\[
\lim_{m\ra\infty} P_\w\left( \frac{ T_{n_{k_m}} - E_\w T_{n_{k_m}} }{ \sqrt{v_{k_m,\w}} } \leq x \right) = \Psi(x+1), \qquad \forall x\in\R,
\]
and 
\[
\lim_{m\ra\infty} P_\w \left( \frac{X_{t_m} - n_{k_m}}{v_P \sqrt{v_{k_m,\w}} } \leq x \right) = 1-\Psi(-x+1), \qquad \forall x\in\R,
\] 
where $t_m=t_m(\w):= \left\lfloor E_\w T_{n_{k_m}} \right\rfloor$.
\end{thm}
\noindent\textbf{Remarks:}\\
\textbf{1.} Note that Theorems \ref{qCLT} and \ref{qEXP}
preclude the possiblity of quenched analogues of the annealed
statements \eqref{annealedstableT} and \eqref{annealedstableX}. \\
\textbf{2.} The choice of Gaussian and exponential distributions in Theorems \ref{qCLT} and \ref{qEXP} are the two extremes of what quenched limiting distributions can be found along random subsequences. In fact, it will be shown in Corollary \ref{explimit} that $T_n$ is approximately the sum of a finite number of exponential random variables with random (depending on the environment) parameters. Thus, we expect in fact that any distribution which is the sum of (or limit of sums of) exponential random variables can be acheived as a quenched limiting distribution of $T_n$ along a random subsequence. \\
\textbf{3.} The sequence $n_k=2^{2^k}$ in Theorems \ref{qCLT} and \ref{qEXP} is chosen only for convenience. In fact, for any sequence $n_k$ growing sufficiently fast, $P-a.s.$ there will be a random subsequence $n_{k_m}(\w)$ such that the conclusions of Theorems \ref{qCLT} and \ref{qEXP} hold. \\
\textbf{4.} The definition of $v_{k_m,\w}$ is given below in \eqref{dkvkdef}, and similar to Theorem \ref{Varstable}, it can be shown that 
$\lim_{n\ra\infty} P\left( n_k^{-2/s} v_{k,\w} \leq x \right) = L_{\frac{s}{2},b}(x)$ for some $b>0$. Also, from \eqref{XTLLN} we have that $t_m \sim \E T_1 n_{k_m}$. Thus, the scaling in Theorems \ref{qCLT} and \ref{qEXP} is of the same order as the annealed scaling but cannot be replaced by a deterministic scaling.

As in \cite{pzSL1}, define the ``ladder locations'' $\nu_i$ of the environment by
\begin{align}
\nu_0 = 0, \quad\text{and}\quad \nu_i =
\begin{cases}
\inf\{n > \nu_{i-1}: \Pi_{\nu_{i-1},n-1} < 1\}, &\quad  i \geq 1,\\
\sup \{j < \nu_{i+1}: \Pi_{k,j-1}<1,\quad \forall k<j \}, &\quad  i \leq -1
\,.\end{cases}
\label{nudef}
\end{align}
Throughout the remainder of the paper we will let $\nu=\nu_1$. 
We will sometimes refer to sections of the environment between
$\nu_{i-1}$ and $\nu_i -1$ as ``blocks'' of the environment. Note
that the block between $\nu_{-1}$ and $\nu_0 -1$ is different from
all the other blocks between consecutive ladder locations (in particular it can be that $\Pi_{\nu_{-1},\nu_0-1} \geq 1$), and that all the other blocks have the same distribution as the block from $0$ to $\nu-1$.  As in \cite{pzSL1} we define
the measure $Q$ on environments by $Q(\cdot):=P(\cdot\,|\mathcal{R})$, where
\[
\mathcal{R}:=\{ \w\in\Omega: \Pi_{-k,-1} < 1,\quad\forall k \geq 1\} = \left\{ \w \in \Omega: \sum_{i=-k}^{-1} \log \rho_i < 0, \quad \forall k \geq 1 \right\}.
\]
Note that $P(\mathcal{R}) > 0$ since $E_P \log \rho < 0$. 
$Q$ is defined so that the blocks of the environment between ladder locations are i.i.d.
under $Q$, all with distribution the same as that of the block
from $0$ to $\nu -1$ under $P$. In particular $P$ and $Q$ agree on $\s( \w_i: i\geq 0)$. 

For any random variable $Z$, define the quenched variance $Var_\w Z := E_\w (Z-E_\w Z)^2$. 
In \cite[Theorem 1.1]{pzSL1} it was proved that when $s\in(0,1)$, $n^{-1/s} E_\w T_{\nu_n}$ converges in distribution (under $Q$) to a stable distribution of index $s$. 
Correspondingly, when $s<2$ we will prove the following theorem:
\begin{thm} \label{Varstable}
Let Assumptions \ref{essentialasm} and \ref{techasm} hold, and let $s<2$. Then there exists a $b>0$ such that 
\begin{equation}
\lim_{n\ra\infty} Q\left(\frac{Var_\w T_{\nu_n}}{n^{2/s}} \leq x \right) = \lim_{n\ra\infty} Q\left( \frac{1}{n^{2/s}} \sum_{i=1}^n \left( E_\w^{\nu_{i-1}} T_{\nu_i} \right)^2 \leq x \right) = L_{\frac{s}{2},b}(x) \, .  \label{stableET2}
\end{equation}
\end{thm}
\noindent\textbf{Remarks:} \\
\textbf{1.} The constant $b$ in the above theorem may not be the same as in \eqref{annealedstableT} and \eqref{annealedstableX}. \\
\textbf{2.} Theorem \ref{Varstable} can be used to show that $\lim_{n\ra\infty} P\left(\frac{Var_\w T_{n}}{n^{2/s}} \leq x \right) = L_{\frac{s}{2},b'}(x)$ for some $b'>0$, but we will not prove this since we do not use it for the other results in this paper.

A major difficulty in analyzing $T_{\nu_n}$ is that the crossing time from $\nu_{i-1}$ to $\nu_i$ depends on the entire environment to the left of $\nu_i$. Thus $Var_\w (T_{\nu_i} - T_{\nu_{i-1}})$ and $Var_\w (T_{\nu_j} - T_{\nu_{j-1}})$ are not independent even if $|i-j|$ is large. 
In order to make the crossing times of blocks that are far apart essentially
independent, we introduce some reflections to the RWRE. For
$n=1,2,\ldots$, define
\begin{equation}
b_n:= \lfloor \log^2(n) \rfloor. \label{bdef}
\end{equation}
Let $\bar{X}_t^{(n)}$ be the random walk that is the same as $X_t$
with the added condition that after reaching $\nu_k$ the
environment is modified by setting $\w_{\nu_{k-b_n}} = 1 $ (i.e.
never allow the walk to backtrack more than $\log^2(n)$ blocks). 
We couple $\bar{X}_t^{(n)}$ with the random walk $X_t$ in such a way that $\bar{X}_t^{(n)} \geq X_t$ with equality holding until
the first time $t$ when the walk $\bar{X}_t^{(n)}$ reaches a modified environment location.
Denote by $\bar{T}_{x}^{(n)}$ the corresponding hitting
times for the walk $\bar{X}_t^{(n)}$. 
It was shown in \cite[Lemma 4.5]{pzSL1} that $\lim_{n\ra\infty} P_\w( T_{\nu_n} \neq \bar{T}_{\nu_n}^{(n)} ) = 0$, $P-a.s.$ so that in fact with high probability the added reflections do not affect the walk at all before $T_{\nu_n}$. 
For ease of notation we define 
\[
\mu_{i,n,\w} := E_\w^{\nu_{i-1}} \bar{T}^{(n)}_{\nu_i}, \quad\text{and}\quad \s_{i,n,\w}^2 := Var_\w \left( \bar{T}^{(n)}_{\nu_i} - \bar{T}^{(n)}_{\nu_{i-1}} \right). 
\]

The structure of the paper is as follows:
In Section \ref{gp} we prove the following general proposition that allows us to easily transfer quenched limit laws from subsequences of $T_n$ to $X_n$.
\begin{prop} \label{generalprop}
Let Assumptions \ref{essentialasm} and \ref{techasm} hold, and let $s\in(1,2)$. Also, let $n_k$ be a sequence of integers growing fast enough so that $\lim_{k\ra\infty} \frac{n_k}{n_{k-1}^{1+\d}} = \infty$ for some $\d>0$, and define
\begin{equation}
d_k:= n_k-n_{k-1},\quad\text{and}\quad v_{k,\w} := \sum_{i=n_{k-1}+1}^{n_k} \s_{i,d_k,\w}^2 = Var_\w \left( \bar{T}^{(d_{k})}_{\nu_{n_k}} - \bar{T}^{(d_{k})}_{\nu_{n_{k-1}}} \right) \, . \label{dkvkdef}
\end{equation}
Assume that $F$ is a continuous distribution function for which $P-a.s.$ there exists a subsequence $n_{k_m}= n_{k_m}(\w)$ such that for $\a_m:= n_{k_m-1}$,
\[
\lim_{m\ra\infty} P_\w^{\nu_{\a_m}}\left( \frac{\bar{T}^{(d_{k_m})}_{x_m} - E_\w^{\nu_{\a_m}}\bar{T}^{(d_{k_m})}_{x_m} }{\sqrt{v_{k_m,\w}}} \leq y \right) = F(y), \quad \forall y\in \R,
\]
for any sequence $x_m \sim n_{k_m}$. Then, $P-a.s.$ for all $y\in \R$ we also have
\begin{equation}
\lim_{m\ra\infty} P_\w\left( \frac{T_{x_m} - E_\w T_{x_m} }{\sqrt{v_{k_m,\w}}} \leq y \right) = F(y), \label{Tlim}
\end{equation}
for any $x_m\sim n_{k_m}$, and 
\begin{equation}
\lim_{m\ra\infty} P_\w\left( \frac{X_{t_m} - n_{k_m} }{ v_P\sqrt{v_{k_m,\w}}} \leq y \right) = 1-F(-y),   \label{Xlim}
\end{equation}
where $t_m:= \left\lfloor E_\w T_{n_{k_m}} \right\rfloor$. 
\end{prop}
Then in Sections \ref{qGauss} and \ref{exponential} we use Theorem \ref{Varstable} to find subsequences $n_{k_m}(\w)$ that allow us to apply Proposition \ref{generalprop}. To find a subsequence that gives Gaussian behavior of $T_{n_{k_m}}$ we find a subsequence where none of the crossing times of the first $n_{k_m}$ blocks is too much larger than all the others and then use the Linberg-Feller condition for triangular arrays. In contrast, to find a subsequence that gives exponential behavior of $T_{n_{k_m}}$ we first prove that the crossing times of ``large'' blocks is approximately exponential in distribution. Then we find a subsequence where the crossing time of one of the first $n_{k_m}$ blocks dominates the total crossing time of the first $n_{k_m}$ blocks. Finally, Section \ref{qvs} contains the proof of Theorem \ref{Varstable} which is similar to that of \cite[Theorem 1.1]{pzSL1}.

Before continuing with the proofs of the main theorems we recall some notation and results from \cite{pzSL1} that will be used throughout the paper.
First, recall that from \cite[Lemma 2.1]{pzSL1} there exist constants $C_1,C_2>0$ such that 
\begin{equation}
P(\nu > x) \leq C_1 e^{-C_2 x}, \quad \forall x\geq 0. \label{nutail}
\end{equation}
Then, since $\nu_n = \sum_{i=1}^n \nu_i - \nu_{i-1}$ and the $\nu_i-\nu_{i-1}$ are i.i.d., the law of large numbers gives that
\begin{equation}
\lim_{n\ra\infty} \frac{\nu_n}{n} = E_P \nu =: \bar \nu < \infty, \qquad P-a.s. \label{nuLLN}
\end{equation}
In \cite{pzSL1} the following formulas for the quenched expectation and variance of $T_\nu$ were given:
\begin{equation}
E_\w T_\nu = \nu + 2 \sum_{j=0}^{\nu-1} W_j, \quad\text{and}\quad Var_\w T_\nu = 4\sum_{j=0}^{\nu-1}(W_{j}+W_{j}^2) + 8\sum_{j=0}^{\nu-1}\sum_{i< j} \Pi_{i+1,j}(W_{i}+W_{i}^2). \label{qvformula}
\end{equation}
Note that since the added reflections only decrease crossing times we obviously have $T_\nu \geq \bar{T}^{(n)}_\nu$ and $E_\w T_\nu \geq E_\w \bar{T}^{(n)}_\nu$ for any $n$. Also, since \eqref{qvformula} holds for any environment $\w$, the formula for $Var_\w \bar{T}^{(n)}_\nu$ is the same as in \eqref{qvformula} but with $\rho_{\nu_{-b_n}}$ replaced by 0. In particular, this shows that $Var_\w T_\nu \geq Var_\w \bar{T}^{(n)}_\nu$ for any $n$. As in \cite{pzSL1} define for any integer $i$
\begin{equation}
M_i:=\max \{ \Pi_{\nu_{i-1}, j} : j\in[\nu_{i-1},\nu_i) \} \, . \label{Mdef}
\end{equation} 
Then \cite[Theorem 1]{iEV} gives that there exists a constant $C_3<\infty$ such that 
\begin{equation}
Q(M_i > x) = P(M_1 > x) \sim C_3 x^{-s}. 	\label{Mtail}
\end{equation}
Note that $M_1 \leq \max_{0\leq j<\nu} W_j$. Therefore,
from the formulas for $E_\w T_\nu$ and $Var_\w T_\nu$ in \eqref{qvformula} it is easy to see that $E_\w T_\nu  \geq M_1$ and $Var_\w T_\nu  \geq M_1^2$ (the same also being true with $\bar{T}_\nu^{(n)}$). 
Finally, recall the following results from \cite{pzSL1}:
\begin{thm}[\textbf{Lemma 3.3 \& Theorem 5.1 in \cite{pzSL1}}] \label{VETtail}
There exists a constant $K_\infty \in (0,\infty)$ such that 
\[
Q\left( Var_\w T_\nu > x \right) \sim Q\left( (E_\w T_\nu)^2 > x \right) \sim K_\infty x^{-s/2}, \qquad \text{as } x\ra\infty.
\]
Moreover, for any $\e>0$ and $x>0$ 
\begin{align*}
Q\left( Var_\w \bar{T}^{(n)}_\nu > x n^{2/s}, \: M_1 > n^{(1-\e)/s} \right) \sim Q\left( \left(E_\w \bar{T}^{(n)}_\nu\right)^2 > x n^{2/s}, \: M_1 > n^{(1-\e)/s} \right) \sim K_\infty x^{-s/2} \frac{1}{n}, 
\end{align*}
as $n\ra\infty$. 
%
\end{thm}

\end{section}

\begin{section}{Converting Time Limits to Space Limits}\label{gp}
In this section we develop a general method for transferring a quenched limit law for a subsequence of $T_n$ to a quenched limit law for a subsequence of $X_n$. We begin with some lemmas analyzing the a.s. asymptotic behavior of the quenched variance and mean of the hitting times.
\begin{lem} \label{Vmdublb}
Assume $s\leq 2$. Then for any $\d > 0$, 
\[
Q\left( Var_\w \bar{T}_{\nu_n}^{(n)} \notin \left(n^{2/s - \d}, n^{2/s+\d}\right) \right) \leq \frac{1}{P(\mathcal{R})} P\left( Var_\w \bar{T}_{\nu_n}^{(n)} \notin \left(n^{2/s - \d}, n^{2/s+\d}\right) \right) = o\left( n^{-\d s/4} \right)  \,.
\]
\end{lem}
\begin{proof}
The first inequality in the lemma is trivial since for any $A\in\mathcal{F}$ we have from the definition of $Q$ that $Q(A)= \frac{P(A\cap \mathcal{R})}{P(\mathcal{R})} \leq \frac{P(A)}{P(\mathcal{R})}$. Next, note that when $s\leq 2$ \cite[Lemma 5.11]{pzSL1} gives
\begin{equation}
P\left( Var_\w \bar{T}_{\nu_n}^{(n)} \geq n^{2/s+\d} \right) \leq P\left( Var_\w T_{\nu_n} \geq n^{2/s + \d} \right) = o( n^{-\d
s/4} )\,. \label{nrub}
\end{equation}
Also, since $Var_\w (\bar{T}_{\nu_{i}}^{(n)} - \bar{T}_{\nu_{i-1}}^{(n)}) \geq M_i^2$ we have 
\[
P\left( Var_\w \bar{T}_{\nu_n}^{(n)} \leq n^{2/s - \d} \right) \leq P\left( M_1^2 \leq n^{2/s-\d}\right)^n = \left( 1 - P\left( M_1 > n^{1/s-\d/2} \right) \right)^n = o\left( e^{-n^{\d s/4}} \right) \, ,
\] 
where the last equality is from \eqref{Mtail}.
\end{proof}
\begin{cor}\label{vkasym}
Assume $s\leq 2$. Then for any $\d>0$
\[
P\left( v_{k,\w} \notin \left( d_k^{2/s -\d}, d_k^{2/s + \d} \right) \right) = o\left( d_k^{-\d s /4} \right).
\]
Consequently, if $s<2$ we have $\sqrt{v_{k,\w}} = o(d_k)$, $P-a.s.$ 
\end{cor}
\begin{proof}
Recall from \eqref{dkvkdef} that by definition $v_{k,\w} = Var_\w \left( \bar{T}_{\nu_{n_k}}^{(d_k)} - \bar{T}_{\nu_{n_{k-1}}}^{(d_k)} \right)$. Also, note that the conditions on $n_k$ ensure that $n_k$ grows faster than exponentially and that $d_k\sim n_k$. Thus, for all $k$ large enough $v_{k,\w}$ only depends on the environment to the right of zero. Therefore for all $k$ large enough
\begin{align*}
P\left( v_{k,\w} \notin \left( d_k^{2/s -\d}, d_k^{2/s + \d} \right) \right) &= Q\left( Var_\w \left( \bar{T}_{\nu_{n_k}}^{(d_k)} - \bar{T}_{\nu_{n_{k-1}}}^{(d_k)} \right) \notin \left( d_k^{2/s -\d}, d_k^{2/s + \d} \right) \right) \\
&= Q\left( Var_\w  \bar{T}_{\nu_{d_k}}^{(d_k)} \notin \left( d_k^{2/s -\d}, d_k^{2/s + \d} \right) \right) = o\left( d_k^{-\d s /4} \right),
\end{align*}
where the last equality is from Lemma \ref{Vmdublb}. Now, for the second claim in the corollary, first note that $2 > \frac{2}{s} + \frac{s-1}{s}$ since $s>1$. Therefore, for any $\e>0$ and for all $k$ large enough we have 
\[
P\left( v_{k,\w} > \e d_k^2 \right) \leq P\left( v_{k,\w} >  d_k^{2/s+(s-1)/s} \right) = o\left( d_k^{-(s-1)/4} \right).
\]
This last term is summable since $d_k$ grows faster than exponentially. Thus the Borel-Cantelli Lemma gives that $v_{k,\w} = o(d_k^2)$, $P-a.s.$
\end{proof}

\begin{cor} \label{begVar}
Assume $s\leq 2$. Then 
\[
\lim_{k\ra\infty} \frac{Var_\w T_{\nu_{n_{k-1}}} }{ v_{k,\w}  } = 0, \quad P-a.s.
\]
\end{cor}
\begin{proof}
By the Borel-Cantelli Lemma it is enough to prove that for any $\e>0$ 
\[
\sum_{k=1}^{\infty} P\left( Var_\w T_{\nu_{n_{k-1}}} \geq \e v_{k,\w} \right) < \infty
\]
However, for any $\d>0$ we have
\begin{equation}
P\left( Var_\w T_{\nu_{n_{k-1}}} \geq \e v_{k,\w} \right) \leq P\left( Var_\w T_{\nu_{n_{k-1}}} \geq \e d_k^{2/s - \d} \right) + P\left( v_{k,\w} \leq d_k^{2/s-\d} \right). \label{Varmd}
\end{equation}
By Corollary \ref{vkasym} the last term in \eqref{Varmd} is summable for any $\d>0$. To show that the second to last term in \eqref{Varmd} is also summable first note that the conditions on the sequence $n_k$ give that there exists a $\d >0$ such that $\e d_k^{2/s - \d} \geq n_{k-1}^{2/s+\d}$ for all $k$ large enough. Thus, for some $\d>0$ and all $k$ large enough we have
\[
P\left( Var_\w T_{\nu_{n_{k-1}}} > \e d_k^{2/s + \d} \right) \leq P\left( Var_\w T_{\nu_{n_{k-1}}} > n_{k-1}^{2/s - \d} \right) = o(n_{k-1}^{-\d s/ 4}),
\]
where the last equality is from \eqref{nrub}. 
\end{proof}
\begin{lem} \label{ETaverage}
Assume $s\in(1,2)$. Then 
$\E T_1 <\infty$, and $P-a.s.$
\begin{equation}
\lim_{k\ra\infty} \frac{E_\w T_{n_k+\lceil x \sqrt{v_{k,\w}} \rceil } - E_\w T_{n_k}}{ \sqrt{v_{k,\w}}} 
= x \E T_1, \quad \forall x\in\R. \label{ETavg}
\end{equation}
\end{lem}
\begin{proof}
Now, since $\frac{E_\w T_{n_k+\lceil x \sqrt{v_{k,\w}} \rceil } - E_\w T_{n_k}}{ \sqrt{v_{k,\w}}}$ is monotone in $x$ it is enough to prove that for arbitrary $x\in\Q$ the limiting statement in \eqref{ETavg} holds. Obviously this is true when $x=0$ since both sides are zero. For the remainder of the proof we'll assume $x>0$. The proof for $x<0$ is essentially the same (recall that by Corollary \ref{vkasym} $v_{k,\w}= o(d_k)=o(n_k)$ when $s<2$). Note that for $x\geq 0$ then we can re-write $E_\w T_{n_k+\lceil x \sqrt{v_{k,\w}} \rceil } - E_\w T_{n_k}= E_\w^{n_k} T_{n_k + \lceil x \sqrt{v_{k,\w}} \rceil }$. By the Borel-Cantelli Lemma it is enough to show that for any $\e>0$,
\begin{equation}
\sum_{k=1}^\infty P\left( \left| E_\w^{n_k} T_{n_k+\lceil x \sqrt{v_{k,\w}} \rceil }  - \lceil x \sqrt{v_{k,\w}} \rceil  \E T_1   \right| \geq \e  \sqrt{v_{k,\w}}  \right) < \infty \, . \label{bc}
\end{equation}
However, for any $\d>0$ we have
\begin{align}
&P\left( \left| E_\w^{n_k} T_{n_k+\lceil x \sqrt{v_{k,\w}} \rceil } - \lceil x \sqrt{v_{k,\w}} \rceil  \E T_1   \right| \geq \e\sqrt{v_{k,\w}}  \right) \nonumber \\
&\:\: \leq P\left( \exists m \in \left[ \lceil x d_k^{1/s-\d} \rceil, \lceil x d_k^{1/s+\d} \rceil \right]: \left| E_\w^{n_k} T_{n_k+ m  } -  m \E T_1   \right| \geq \frac{\e m}{x} \right) + P\left( v_{k,\w} \notin \left[d_k^{2/s-2\d}, d_k^{2/s+2\d}\right] \right) \nonumber \\
&\:\: \leq P\left( \max_{m\leq \lceil x d_k^{1/s+\d} \rceil } \left| E_\w T_m - m\E T_1 \right| \geq \e  d_k^{1/s-\d}  \right) + o( d_k^{-\d s/2} ), \label{rv}
\end{align}
where the last inequality is due to Corollary \ref{vkasym} and the fact that $\{E_\w^{n_k} T_{n_k+m}\}_{m\in\Z}$ has the same distribution as $\{ E_\w T_m \}_{m\in \Z}$ since $P$ is a product measure. 
Thus, we only need to show that the first term in \eqref{rv} is summable in $k$ for some $\d>0$. For this, we need the following lemma whose proof we defer.
\begin{lem} \label{maxLD}
Assume $s\in(1,2]$. Then for any $0<\d' < \frac{s-1}{2s}$ we have that
\[
P\left( \max_{m\leq n} \left| E_\w T_m - m \E T_1 \right| \geq n^{1-\d'} \right) =o \left( n^{-(s-1)/2} \right)  
\]
\end{lem}
Assuming Lemma \ref{maxLD}, fix $0 < \d' < \frac{s-1}{2s}$ and then choose $0<\d < \frac{\d'}{s(2-\d')}$. We choose $\d$ and $\d'$ this way to ensure that $(1/s+\d)(1-\d') < 1/s-\d$. Therefore, for all $k$ large enough, $ \e  d_k^{1/s-\d}  > \left\lceil x d_k^{1/s+\d} \right\rceil^{1-\d'}$. Thus for all $k$ large enough we have
\begin{align*}
P\left( \max_{m\leq \lceil x d_k^{1/s+\d}\rceil } \left| E_\w T_m - m\E T_1 \right| \geq \e  d_k^{1/s-\d}  \right) &\leq P\left( \max_{m\leq \lceil x d_k^{1/s+\d}\rceil} \left| E_\w T_m - m \E T_1\right| \geq \left\lceil x d_k^{1/s+\d} \right\rceil^{1-\d'} \right)\\
&=   o\left( d_k^{-(1/s+\d)(s-1)/2} \right), \qquad \text{as } k\ra\infty.
\end{align*}
Since $s>1$ this last term is summable in $k$. 
\end{proof}
\begin{proof}[Proof of Lemma \ref{maxLD}:]
Before proceeding with the proof we need to introduce some notation for a slightly different type of reflection. Define $\tilde{X_t}^{(n)}$ to be the RWRE modified so that it cannot backtrack a distance of $b_n$ (the definition of $\bar{X_t}^{(n)}$ is similar except the walk was not allowed to backtrack $b_n$ blocks instead). That is, after the walk first reaches location $i$, we modify the environment by setting $\w_{i-b_n}=1$. Let $\tilde{T_x}^{(n)}$ be the corresponding hitting times of the walk $\tilde{X_t}^{(n)}$. Then
\begin{align}
P\left( \max_{m\leq n} \left| E_\w T_m - m \E T_1 \right| \geq n^{1-\d'} \right)
& \leq  P\left( E_\w T_n - E_\w \tilde{T}_n^{(n)}  \geq \frac{n^{1-\d'}}{3} \right)  
+  P\left( \E T_1 - \E \tilde{T}_1^{(n)}  \geq \frac{n^{-\d'}}{3} \right)\nonumber \\
&\quad +  P\left( \max_{m\leq n} \left| E_\w \tilde{T}_m^{(n)} - m \E \tilde{T}_1^{(n)} \right| \geq \frac{n^{1-\d'}}{3} \right) \nonumber \\
& \leq 3 n^{-1+\d'}  (\E T_n - \E \tilde{T}_n^{(n)})  
+ \mathbf{1}_{\E T_1 - \E \tilde{T}_1^{(n)} \geq n^{-\d'}/3} \nonumber \\
&\quad + P\left( \max_{m\leq n} \left| E_\w \tilde{T}_m^{(n)} - m \E \tilde{T}_1^{(n)} \right| \geq \frac{n^{1-\d'}}{3} \right) \label{aref} 
\end{align}
Now, from \eqref{QET} we get that $ E_\w T_1 - E_\w \tilde{T}_1^{(n)}  =  (1+2W_0)-(1+2W_{-b_n+1,0}) = 2 \Pi_{-b_n+1,0} W_{-b_n} $, and thus since $P$ is a product measure
\begin{equation}
\E T_n - \E \tilde{T}_n^{(n)} =  n E_P \left( E_\w T_1 - E_\w \tilde{T}_1^{(n)} \right) = \frac{ 2 n }{1-E_P \rho} (E_P \rho)^{b_n+1}. \label{ett}
\end{equation}
Since $E_P \rho < 1$ and $b_n \sim \log^2 n$ the above decreases faster than any power of $n$. Thus by \eqref{aref} we need only to show that $P\left( \max_{m\leq n} \left| E_\w \tilde{T}_m^{(n)} - m \E \tilde{T}_1^{(n)} \right| \geq \frac{n^{1-\d'}}{3} \right) = o(n^{-(s-1)/2})$. 
For ease of notation we define $\kappa_{i}^{(n)}:=E_\w^{i-1} \tilde{T}_i^{(n)} - \E \tilde{T}_1^{(n)}$. Thus, since $E_\w \tilde{T}_m^{(n)} - m \E \tilde{T}_1^{(n)} = \sum_{i=1}^m \kappa_{i}^{(n)} = \sum_{i=1}^{b_n} \sum_{j=0}^{\left\lfloor \frac{m-i}{b_n} \right\rfloor} \kappa_{jb_n+i}^{(n)} $ we have
\begin{align}
P\left( \max_{m\leq n} \left| E_\w \tilde{T}_m^{(n)} - m \E \tilde{T}_1^{(n)} \right| \geq \frac{n^{1-\d'}}{3} \right) 
&\leq P\left(  \max_{m\leq n} \sum_{i=1}^{b_n} \left| \sum_{j=0}^{\left\lfloor \frac{m-i}{b_n} \right\rfloor} \kappa_{jb_n+i}^{(n)} \right| \geq \frac{n^{1-\d'}}{3}  \right) \nonumber \\
&\leq \sum_{i=1}^{b_n} P\left( \max_{m\leq n}  \left| \sum_{j=0}^{\left\lfloor \frac{m-i}{b_n} \right\rfloor} \kappa_{jb_n+i}^{(n)} \right| \geq \frac{n^{1-\d'}}{3 b_n}  \right) \nonumber \\
&= \sum_{i=1}^{b_n} P\left( \max_{l\leq \left\lfloor \frac{n-i}{b_n} \right\rfloor}  \left| \sum_{j=0}^{l} \kappa_{jb_n+i}^{(n)} \right| \geq \frac{n^{1-\d'}}{3 b_n}  \right). \label{depldb}
\end{align}
Due to the reflections of the random walk, $\kappa_i^{(n)}$ depends only on the environment between $i-b_n$ and $i-1$. Thus, for each $i$ $\{\kappa_{jb_n+i}^{(n)}\}_{j=0}^{\infty}$ is a sequence of $i.i.d.$ random variables with zero mean, and so $\{\sum_{j=0}^l \kappa_{jb_n+i}^{(n)}\}_{l\geq 0}$ is a martingale. Now, let $\gamma \in (1,s)$. Then, by the Doob-Kolmogorov inequality, for any integer $N$ we have
\[
P\left( \max_{l\leq N}  \left| \sum_{j=0}^{l} \kappa_{jb_n+i}^{(n)} \right| \geq \frac{n^{1-\d'}}{3 b_n}  \right) \leq 3^\gamma b_n^\gamma n^{-\gamma+\gamma\d'} E_P\left| \sum_{j=0}^N \kappa_{jb_n+i}^{(n)}  \right|^\gamma.
\]
Now, since $\{\kappa_{jb_n+i}^{(n)}\}_{j=0}^{\infty}$ is a sequence of independent, zero-mean random variables, the Marcinkiewicz-Zygmund inequality \cite[Theorem 2 on p. 356]{ctMZ} gives that there exists a constant $B_\gamma<\infty$ depending only on $\gamma > 1$ such that 
\[
E_P\left| \sum_{j=0}^N \kappa_{jb_n+i}^{(n)}  \right|^\gamma \leq B_\gamma E_P \left| \sum_{j=0}^N \left(\kappa_{jb_n +i}^{(n)}\right)^2 \right|^{\gamma/2} \leq B_\gamma E_P \left( \sum_{j=0}^N \left| \kappa_{jb_n +i}^{(n)}\right|^\gamma \right) = B_\gamma (N+1) E_P|\kappa_1^{(n)}|^\gamma, 
\]
where the second inequality is because $\gamma < s \leq 2$ implies $\gamma/2 < 1$. 
Now, recall from \cite{kksStable} that $P(E_\w T_1 > x) \sim K x^{-s}$ for some $K>0$. Therefore, since $\gamma < s$ we have that $E_P |E_\w T_1|^\gamma < \infty$.  Thus, it's easy to see that $E_P|\kappa_1^{(n)}|^\gamma = E_P \left| E_\w \tilde{T}_1^{(n)} - \E \tilde{T}_1^{(n)} \right|^\gamma $ is uniformly bounded in $n$. So, there exists a constant $B_\gamma '$ depending on $\gamma \in (1,s)$ such that
\[
P\left( \max_{l\leq N}  \left| \sum_{j=0}^{l} \kappa_{jb_n+i}^{(n)} \right| \geq \frac{n^{1-\d'}}{3 b_n}  \right) \leq B'_\gamma b_n^\gamma n^{-\gamma+\gamma\d'} (N+1),
\]
and thus by \eqref{depldb}
\[
P\left( \max_{m\leq n} \left| E_\w \tilde{T}_m^{(n)} - m \E \tilde{T}_1^{(n)} \right| \geq \frac{n^{1-\d'}}{3} \right) \leq B'_\gamma b_n^{\gamma+1} n^{-\gamma+\gamma\d'} \left( \frac{n}{b_n}+1 \right) = \bigo\left( b_n^\gamma n^{1-\gamma + \gamma \d'} \right).
\]
Since by assumption we have $\d'<\frac{s-1}{2s}$, we may choose $\gamma < s$ arbitrarily close to $s$ so that $b_n^\gamma n^{-\gamma +1+ \gamma \d'} = o\left(n^{-(s-1)/2}\right)$.
\end{proof}
\begin{proof}[\textbf{Proof of Proposition \ref{generalprop}:}] \ \\
Recall the definition of $\a_m:= n_{k_m-1}$. To prove \eqref{Tlim} it is enough to prove that $\forall \e>0$
\begin{equation}
\lim_{m\ra\infty} P_\w\left( \left| \frac{T_{\nu_{\a_m}} - E_\w T_{\nu_{\a_m}} }{\sqrt{v_{k_m,\w}}} \right| \geq \e \right) = 0,  \quad P-a.s. \label{rmstart}
\end{equation}
and
\begin{equation}
\lim_{m\ra\infty} P_\w^{\nu_{\a_m}} \left( T_{x_m} \neq \bar{T}^{(d_{k_m})}_{x_m} \right) = 0, \quad\text{and}\quad \lim_{m\ra\infty} E_\w^{\nu_{\a_m}} \left( T_{x_m} - \bar{T}^{(d_{k_m})}_{x_m} \right) = 0, \quad P-a.s.  \label{addreflections}
\end{equation}
To prove \eqref{rmstart}, note that by Chebychev's inequality 
\[
P_\w\left( \left| \frac{T_{\nu_{\a_m}} - E_\w T_{\nu_{\a_m}} }{\sqrt{v_{k_m,\w}}} \right| \geq \e \right) \leq \frac{ Var_\w T_{\nu_{\a_m}} }{ \e^2 v_{k_m,\w} },
\]
which by Corollary \ref{begVar} tends to zero $P-a.s.$ as $m\ra\infty$. Secondly, to prove \eqref{addreflections}, note that since
\[
P_\w^{\nu_{\a_m}} \left( T_{x_m} \neq \bar{T}^{(d_{k_m})}_{x_m} \right) = P_\w^{\nu_{\a_m}} \left( T_{x_m} - \bar{T}^{(d_{k_m})}_{x_m} \geq 1 \right) \leq E_\w^{\nu_{\a_m}} \left( T_{x_m} - \bar{T}^{(d_{k_m})}_{x_m} \right),
\]
it is enough to prove only the second claim \eqref{addreflections}. However, since $x_m \leq 2 n_{k_m}$ for all $m$ large enough, it is enough to prove
\begin{equation}
\lim_{k\ra\infty}  E_\w \left( T_{2n_k} - \bar{T}^{(d_k)}_{2n_k} \right) = 0, \quad P-a.s. \label{rdz}
\end{equation}
To prove \eqref{rdz}, note that for any $\e>0$ that 
\begin{equation}
P\left(  E_\w \left( T_{2n_k} - \bar{T}^{(d_k)}_{2n_k} \right)  \geq \e \right) \leq \frac{\E \left( T_{2n_k} - \bar{T}^{(d_k)}_{2n_k} \right)}{\e}  \leq \frac{\E \left( T_{2n_k} - \tilde{T}^{(d_k)}_{2n_k} \right)}{\e}  = \frac{2 n_k \E \left( T_1 - \tilde{T}^{(d_k)}_1 \right)}{\e} . \label{rdz2}
\end{equation}
However, from \eqref{ett} we have that $\E \left( T_1 - \tilde{T}^{(d_k)}_1 \right) = \frac{2}{1-E_P \rho} (E_P \rho)^{b_{d_k}}$ which decreases faster than any power of $n_k$ (since $E_P \rho < 1$ and $d_k\sim n_k$), and thus the last term in \eqref{rdz2} is summable. Therefore, applying the Borel-Cantelli Lemma gives \eqref{rdz} which completes the proof of \eqref{Tlim}. 
Note, moreover, that the convergence in \eqref{Tlim} must be uniform in $y$ since $F$ is continuous.

To prove \eqref{Xlim}, for any $y\in \R$ let $ x_m(y):= \left \lceil n_{k_m}+ y \, v_P \sqrt{v_{k_m,\w}} \right \rceil$, and define 
$X_t^* := \max_{n\leq t} X_n$.
Then we have
\begin{align}
P_\w \left( \frac{X_{t_m}^* - n_{k_m}}{v_P \sqrt{v_{k_m,\w}} } < y \right) &= P_\w \left( X_{t_m}^* < x_m(y) \right) \nonumber = P_\w \left( T_{ x_m(y) } > t_m \right) \nonumber \\
&= P_\w \left( \frac{ T_{ x_m(y) } - E_\w T_{ x_m(y) } }{ \sqrt{v_{k_m,\w}} } > \frac{t_m - E_\w T_{  x_m(y)  } }{\sqrt{v_{k_m,\w}}} \right)  \label{timespace}
\end{align}
Now, recalling the definition of $t_m:= \left\lfloor E_\w X_{n_{k_m}} \right\rfloor$, by Lemma \ref{ETaverage} we have
\[
\lim_{m\ra\infty} \frac{t_m - E_\w T_{ x_m(y) } }{\sqrt{v_{k_m,\w}}} =
\lim_{m\ra\infty} \frac{ \left\lfloor E_\w T_{n_{k_m}} \right\rfloor - E_\w T_{  n_{k_m} + y v_P \sqrt{v_{k_m,\w}}  } }{\sqrt{v_{k_m,\w}}} = -y, \quad \forall y\in \R \quad P-a.s.,
\]
where we used the fact that $v_P \E T_1 = 1$ due to \eqref{XTLLN}. Also, by Corollary \ref{vkasym} we have $P-a.s.$ that $\sqrt{v_{k,\w}} = o(d_k) = o(n_k)$ since $s<2$, and therefore $x_m(y) \sim n_{k_m}$.
Thus since the convergence in \eqref{Tlim} is uniform in $y$, \eqref{timespace} gives that 
\begin{equation}
\lim_{m\ra\infty} P_\w \left( \frac{X_{t_m}^* - n_{k_m}}{v_P \sqrt{v_{k_m,\w}} } < y \right) = 1-F(-y), \quad \forall y\in \R \quad P-a.s. \label{Xtstarlim}
\end{equation}
Now, \eqref{XTLLN} gives that $t_m \sim (\E T_1)n_{k_m}$, $P-a.s.$ Therefore, an easy argument involving \cite[Lemma 4.6]{pzSL1} and \eqref{nutail} gives that $X_{t_m}^*-X_{t_m} = o( \log^2 t_m ) =o(\log^2 n_{k_m})$, $\P-a.s$. Also, Corollary \ref{vkasym} and the Borel-Cantelli Lemma give $P-a.s.$ that $v_{k,\w} \geq d_k^{2/s-\d} \sim n_k^{2/s - \d}$ for any $\d>0$ and all $k$ large enough. Therefore, $\P-a.s.$ we have that $\lim_{m\ra\infty} \frac{ X_{t_m}^* - X_{t_m} }{\sqrt{v_{k_m,\w}}} = 0$. Combining this with \eqref{Xtstarlim} completes the proof of \eqref{Xlim}. 
\end{proof}
\noindent\textbf{Remark:} For the last conclusion of Proposition \ref{generalprop} to hold it is crucial that $s>1$. The dual nature of $X_t^*$ and $T_n$ always allows the transfer of probabilities from time to space. However, if $s\leq 1$ then $\E T_1 = \infty$ and the averaging behavior of Lemma \ref{ETaverage} does not occur. 
\end{section}

\begin{section}{Quenched CLT Along a Subsequence} \label{qGauss}
For the remainder of the paper we will fix the sequence $n_k:= 2^{2^k}$ and let $d_k$ and $v_{k,\w}$ be defined accordingly as in \eqref{dkvkdef}. Note that this choice of $n_k$ satisfies the conditions in Proposition \ref{generalprop} for any $\d < 1$ since $n_k=n_{k-1}^2$. Our first goal in this section is to prove the following theorem, which when applied to Proposition \ref{generalprop} proves Theorem \ref{qCLT}.
\begin{thm}\label{Tngaussian}
Assume $s<2$. Then for any $\eta\in(0,1)$, $P-a.s.$ there exists a subsequence $n_{k_m}=n_{k_m}(\w, \eta)$ of $n_k=2^{2^k}$ such that for $\a_m,\beta_m$ and $\gamma_m$ defined by
\begin{equation}
\a_m:= n_{k_m-1}, \quad \b_m:= n_{k_m-1} + \left\lfloor \eta d_{k_m} \right\rfloor, \quad\text{and}\quad \gamma_m:= n_{k_m} \label{abgdef}
\end{equation}
and any sequence $x_m \in \left(\nu_{\beta_m} , \nu_{\gamma_m} \right]$  we have
\[
\lim_{m\ra\infty} P_\w^{\nu_{\a_m}} \left( \frac{\bar{T}^{(d_{k_m})}_{x_m} - E_\w \bar{T}^{(d_{k_m})}_{x_m} }{\sqrt{v_{k_m,\w}}} \leq x \right) = \Phi(x).  
\]
\end{thm}
The proof of Theorem \ref{qGauss} is similar to the proof of \cite[Theorem 5.10]{pzSL1}. The key is to find a random subsequence where none of the variances $\s_{i,d_{k_m},\w}^2$ with $i\in (n_{k_m-1},n_{k_m}]$ is larger than a fraction of $v_{k_m,\w}$. 
To this end, let $\#(I)$ denote the cardinality of the set $I$, and for any $\eta\in(0,1)$ and any positive integer $a < n/2$
define the events
\[
\mathcal{S}_{\eta,n,a} := \bigcup_{ \tiny{\begin{array}{c} I\subset [1,\eta n] \\ \#(I) = 2a \end{array}}} \!\! \left( \bigcap_{i\in I} \left\{  \mu_{i,n,\w}^2 \in[n^{2/s},2n^{2/s}) \right\} \bigcap_{j\in[1,\eta n]\backslash I} \left\{ \mu_{j,n,\w}^2 < n^{2/s} \right\} \right) 
\,.
\]
and
\[
U_{\eta,n} := \left\{ \sum_{i\in(\eta n, n]} \s_{i,n,\w}^2 < 2 n^{2/s}   \right\} .
\]
On the event $\mathcal{S}_{\eta, n, a}$, $2a$ of the first $\eta
n$ crossings times from $\nu_{i-1}$ to $\nu_i$ have roughly the
same size variance and the rest are
all smaller. Define
\begin{equation}
a_k:=\lfloor\log\log k\rfloor \vee 1. \label{akdef}
\end{equation} 
Then, we have the following Lemma:
\begin{lem}\label{smbklemma}
Assume $s<2$. Then for any $\eta \in (0,1)$, we have $Q\left( \mathcal{S}_{\eta,d_k,a_k} \cap U_{\eta,d_k} \right)\geq \frac{1}{k}$ for all $k$ large enough.
\end{lem}
\begin{proof}
First we reduce the problem to getting a lower bound on $Q(\mathcal{S}_{\eta,d_k,a_k} )$. Define
\[
\tilde{U}_{\eta,n} := \left\{ \sum_{i\in(\eta n + b_n, n]} \s_{i,n,\w}^2 < n^{2/s} \right\}. 
\]
Note that $\mathcal{S}_{\eta,n,a}$ and $\tilde{U}_{\eta,n}$ are independent events since $\tilde{U}_{\eta,n}$ only depends on the environment to the right of the $\nu_{\lceil \eta n \rceil}$. Thus,
\begin{align*}
Q\left( \mathcal{S}_{\eta,n,a} \cap U_{\eta, n}\right) &\geq Q\left( \mathcal{S}_{\eta,n,a} \cap \tilde{U}_{\eta,n} \right) - Q\left( \sum_{i\in(\eta n, \eta n +b_n]} \s_{i,n,\w}^2 > n^{2/s} \right) \\
&\geq Q\left( \mathcal{S}_{\eta,n,a} \right) Q\left( \tilde{U}_{\eta,n} \right) - b_n Q\left( Var_\w \bar{T}^{(n)}_\nu > \frac{n^{2/s}}{b_n} \right).
\end{align*}
Now, Theorem \ref{Varstable} gives that $Q\left( \tilde{U}_{\eta,n} \right) \geq 
Q\left( Var_\w T_{\nu_n} < n^{2/s} \right) = L_{\frac{s}{2},b}(1) + o(1)$, and Theorem \ref{VETtail} gives that $b_n Q\left( Var_\w \bar{T}^{(n)}_\nu > \frac{n^{2/s}}{b_n} \right) \sim K_\infty b_n^{1+s} n^{-1}$. Thus,
\[
Q\left( \mathcal{S}_{\eta,d_k,a_k} \cap U_{\eta, d_k}\right) \geq Q(\mathcal{S}_{\eta,d_k,a_k})( L_{\frac{s}{2},b}(1) + o(1) ) - \bigo(b_{d_k}^{1+s} d_k^{-1}), \quad \text{as }k\ra\infty,
\]
and so to prove the lemma it is enough to show that $\lim_{k\ra\infty} k \, Q(\mathcal{S}_{\eta,d_k,a_k}) = \infty$.  
A lower bound for $Q( \mathcal{S}_{\eta,n,a})$ was derived in \cite[preceeding Lemma 5.7]{pzSL1}. 
A similar argument gives that for any $\e<\frac{1}{3}$ there exists a constant $C_\e>0$ such that
\begin{align}
Q\left( \mathcal{S}_{\eta,n,a} \right) & \geq  \frac{(\eta
C_\e)^{2a}}{(2a)!} \left( 1 - \frac{(2a-1)(1+4b_n)}{\eta n}
\right)^{2a} \left( Q\left( \sum_{i=1}^n \left( E_\w^{\nu_{i-1}} T_{\nu_i} \right)^2 < n^{2/s} \right) - a \, o(n^{-1+2\e}) \right)  \nonumber \\
&\qquad - \frac{(\eta n)^{2a}}{(2a)!} a \, o\left(
e^{-n^{\e/(6s)}} \right) \label{calSlb}
\,,\end{align}
where asymptotics of the form $o(\cdot\,)$ in \eqref{calSlb} are uniform in $\eta$ and $a$ as $n\ra\infty$. The proof of \eqref{calSlb} is exactly the same as in \cite{pzSL1} with the exception that the lower bound for $Q\left( \bigcap_{j\in[1,n]} \left\{ \mu_{j,n,\w}^2 < n^{2/s} \right\} \right)$  in \cite[(70)]{pzSL1} is $Q\left( \sum_{i=1}^n \left( E_\w^{\nu_{i-1}} T_{\nu_{i}} \right)^2 < n^{2/s} \right)$ instead of $Q\left( E_\w T_{\nu_n} < n^{1/s} \right)$. 
Then, replacing $n$ and $a$ in \eqref{calSlb} by $d_k$ and
$a_k$ respectively, we have for $\e<\frac{1}{3}$ that 
\begin{align}
&Q\left( \mathcal{S}_{\eta,d_k,a_k} \right) \nonumber \\
&\qquad \geq  \frac{(\eta
C_\e)^{2a_k}}{(2a_k)!} \left( 1 - \frac{(2a_k-1)(1+4b_{d_k})}{\eta
d_k} \right)^{2a_k} \left( Q\left( \sum_{i=1}^{d_k} \left( E_\w^{\nu_{i-1}} T_{\nu_i} \right)^2  < d_k^{2/s}
\right) - a_k
o(d_k^{-1+2\e}) \right)  \nonumber \\
&\qquad\qquad - \frac{(\eta d_k)^{2a_k}}{(2a_k)!} a_k o\left(
e^{-d_k^{\e/(6s)}} \right) \nonumber \\
&\qquad =  \frac{(\eta C_\e)^{2a_k}}{(2a_k)!} \left( 1+ o(1)\right)
\left( L_{\frac{s}{2},b}(1) - o(1) \right) - o\left( \frac{1}{k} \right). \label{sublb}
\end{align}
The last equality is a result of Theorem \ref{Varstable} and the definitions of $a_k$
and $d_k$ in \eqref{akdef} and \eqref{dkvkdef}. 
Also, since $a_k\sim \log\log k$ we have that $\lim_{k\ra\infty} k \frac{C^{2a_k}}{(2a_k)!} = \infty $ for any constant $C>0$. Therefore, \eqref{sublb} implies that $\lim_{k\ra\infty} k\, Q\left( \mathcal{S}_{\eta,d_k,a_k} \right) = \infty$. 
\end{proof}
\begin{cor}\label{smallblocks2}
Assume $s<2$. Then for any $\eta\in(0,1)$,
$P$-a.s. there exists a random subsequence $n_{k_m}=n_{k_m}(\w,\eta)$
of $n_k=2^{2^k}$ such that for the sequences $\a_m,\b_m,$ and $\gamma_m$ defined as in \eqref{abgdef} we have that for all $m$
\begin{equation}
\max_{ i\in(\a_m, \b_m] } \mu_{i,d_{k_m},\w}^2  \leq
2d_{k_m}^{2/s} \leq \frac{1}{a_{k_m}} \sum_{i=\a_m+1}^{\b_m}
\mu_{i,d_{k_m},\w}^2, \quad\text{and}\quad  \sum_{i=\b_m + 1}^{\gamma_m} \s_{i,d_{k_m},\w}^2 < 2 d_{k_m}^{2/s} .\label{smbk}
\end{equation}
\end{cor}
\begin{proof}
Define the sequence of events
\[
\mathcal{S}_k' := \bigcup_{ \tiny{\begin{array}{c} I\subset (n_{k-1}, n_{k-1}+\eta d_k] \\ \#(I) = 2a_k \end{array}}} \!\! \left( \bigcap_{i\in I} \left\{  \mu_{i,d_k,\w}^2 \in[d_k^{2/s},2d_k^{2/s}) \right\} \bigcap_{j\in(n_{k-1},n_{k-1}+\eta d_k]\backslash I} \left\{ \mu_{j,d_k,\w}^2 < d_k^{2/s} \right\} \right) 
\,,
\]
and
\[
U_k' := \left\{ \sum_{i \in (n_{k-1}+ \eta d_k , n_k]} \s_{i,d_{k_m},\w}^2 < 2 d_{k_m}^{2/s} \right\}
\]
Note that due to the reflections of the random walk, the event
$\mathcal{S}_k' \cap U_k' $ depends on the environment between
ladder locations $n_{k-1}-b_{d_k}$ and $n_k$. Thus, since $n_{k-1} - b_{d_k} > n_{k-2}$ for all $k\geq 4$, we have that 
$\{\mathcal{S}_{2k}' \cap U_{2k}'\}_{k=2}^{\infty}$ is an independent sequence of events.
Similarly, for $k$ large enough $\mathcal{S}_k' \cap U_k'$ does not
depend on the environment to left of the origin. Thus
\[
P(\mathcal{S}_k' \cap U_k') =Q(\mathcal{S}_k' \cap U_k' )=Q\left( \mathcal{S}_{\eta,d_k,a_k} \cap U_{\eta,d_k} \right)
\]
for all $k$ large enough. Lemma \ref{smbklemma} then gives
that $\sum_{k=1}^\infty P(\mathcal{S}_{2k}' \cap U_{2k}') =
\infty$, and the Borel-Cantelli Lemma then implies that infinitely
many of the events $\mathcal{S}_{2k}' \cap U_{2k}' $ occur $P-a.s$.
Therefore, $P-a.s.$ there exists a subsequence $k_m=k_m(\w,\eta)$ such that $\mathcal{S}_{k_m}' \cap U_{k_m}'$ occurs for each $m$. Finally, note that the event $\mathcal{S}_{k_m}' \cap U_{k_m}'$ implies \eqref{smbk}. 
\end{proof}
\begin{proof}[\textbf{Proof of Theorem \ref{Tngaussian}:}] \ \\
First, recall that \cite[Corollary 5.6]{pzSL1} gives that there exists an $\eta'>0$ such that 
\begin{equation}
Q\left( \left| \sum_{i=1}^n \left( \s_{i,m,\w}^2-\mu_{i,m,\w}^2
\right) \right| \geq \d n^{2/s} \right) = o(n^{-\eta'}) \qquad \forall \d>0, \quad \forall m\in \N. \label{c56}
\end{equation}
This can be applied along with the Borel-Cantelli Lemma to prove that
\begin{equation}
\sum_{i=n_{k-1}+1}^{n_{k-1} + \lfloor \eta d_k \rfloor } \!\!\! \left( \s_{i,d_k,\w}^2 - \mu_{i,d_k,\w}^2 \right) =  o\left(d_k^{2/s}\right), \quad P-a.s. \label{vmucompare}
\end{equation}
Thus, $P-a.s.$ we may assume that \eqref{vmucompare} holds and that there exists a subsequence $n_{k_m}=n_{k_m}(\w,\eta)$ such that condition \eqref{smbk} in Corollary \ref{smallblocks2} holds. Then, it is enough to prove that
\begin{equation}
\lim_{m\ra\infty} P_\w^{\nu_{\a_m}} \left( \frac{\bar{T}^{(d_{k_m})}_{\nu_{\b_m}} - E_\w^{\nu_{\a_m}} \bar{T}^{(d_{k_m})}_{\nu_{\b_m}}}{\sqrt{v_{k_m,\w}}} \leq y \right) = \Phi(y), \label{reducegaussian}
\end{equation}
and
\begin{equation}
\lim_{m\ra\infty} P_\w^{\nu_{\b_m}} \left( \left| \frac{\bar{T}^{(d_{k_m})}_{x_m} - E_\w^{\nu_{\b_m}} \bar{T}^{(d_{k_m})}_{x_m}}{\sqrt{v_{k_m,\w}}} \right| \geq \e \right) = 0, \quad \forall \e>0. \label{endsmall}
\end{equation}
To prove \eqref{endsmall}, note that by Chebychev's inequality
\[
P_\w^{\nu_{\b_m}} \left( \left| \frac{\bar{T}^{(d_{k_m})}_{x_m} - E_\w^{\nu_{\b_m}} \bar{T}^{(d_{k_m})}_{x_m}}{\sqrt{v_{k_m,\w}}} \right| \geq \e \right) \leq \frac{Var_\w \left( \bar{T}^{(d_{k_m})}_{x_m} - \bar{T}^{(d_{k_m})}_{\b_m} \right) }{\e^2 v_{k_m,\w}} \leq \frac{\sum_{i=\b_m+1}^{\gamma_m} \s_{i,d_{k_m},\w}^2 }{\e^2 v_{k_m,\w}} 
\]
However, by \eqref{vmucompare} and our choice of the subsequence $n_{k_m}$ we have that $\sum_{i=\b_m+1}^{\gamma_m} \s_{i,d_{k_m},\w}^2 < 2 d_{k_m}^{2/s}$, and $v_{k_m,\w} \geq \sum_{i=\a_m+1}^{\b_m} \s_{i,d_{k_m},\w}^2 = \sum_{i=\a_m+1}^{\b_m} \mu_{i,d_{k_m},\w}^2 + o\left( d_{k_m}^{2/s} \right)\geq a_{k_m} d_{k_m}^{2/s} + o\left( d_{k_m}^{2/s} \right)$. Thus
\begin{equation}
\lim_{m\ra\infty} \frac{\sum_{i=\b_m+1}^{\gamma_m} \s_{i,d_{k_m},\w}^2 }{v_{k_m,\w}} = 0, \label{s2v}
\end{equation}
which proves \eqref{endsmall}.
To prove \eqref{reducegaussian}, it is enough to show that the Lindberg-Feller condition is satisfied. That is we need to show
\begin{equation}
\lim_{m\ra\infty} \frac{1}{v_{k_m,\w}} \sum_{i=\a_m+1}^{\b_m} \s_{i,d_{k_m},\w}^2 = 1, \label{LF1}
\end{equation}
and 
\begin{equation}
\lim_{m\ra\infty} \frac{1}{v_{k_m,\w}} \sum_{i=\a_m+1}^{\b_m} E_\w^{\nu_{i-1}} \left[ \left( \bar{T}^{(d_{k_m})}_{\nu_i} - \mu_{i,d_{k_m},\w} \right)^2 \mathbf{1}_{ | \bar{T}^{(d_{k_m})}_{\nu_i}-\mu_{i,d_{k_m},\w}  | > \e \sqrt{v_{m,\w}}}  \right] = 0, \quad \forall \e>0. \label{LF2}
\end{equation}
To show \eqref{LF1} note that the definition of $v_{k_m,\w}$ and our choice of the subsequence $n_{k_m}$ give that
\[
\frac{1}{v_{k_m,\w}} \sum_{i=\a_m+1}^{\b_m} \s_{i,d_{k_m},\w}^2 = 1 - \frac{1}{v_{k_m,\w}} \sum_{i=\b_m+1}^{\gamma_m} \s_{i,d_{k_m,\w}}^2 = 1-o(1),
\]
where the last equality is from \eqref{s2v}. To prove \eqref{LF2}, first note that an application of \cite[Lemma 5.5]{pzSL1} gives that for any $\e'>0$
\[
\sum_{i=n_{k-1}+1}^{n_{k-1}+\lfloor \eta d_k \rfloor} \s_{i,d_k,\w}^2 \mathbf{1}_{M_i \leq d_k^{(1-\e')/s}} = o\left( d_k^{2/s} \right) ,\quad P-a.s. , 
\]
where $M_i$ is defined as in \eqref{Mdef}. 
Then, since $v_{k_m,\w} \geq a_{k_m} d_{k_m}^{2/s} + o\left( d_{k_m}^{2/s} \right)$ we can reduce the sum in \eqref{LF2} to blocks where $M_i > d_{k_m}^{(1-\e')/s}$. That is, it is enough to prove that for some $\e'>0$ and every $\e>0$
\begin{align}
\lim_{m\ra\infty} \frac{1}{v_{k_m,\w}} \sum_{i=\a_{m}+1}^{\b_m} E_\w^{\nu_{i-1}} \left[ \left( \bar{T}^{(d_{k_m})}_{\nu_i}-\mu_{i,d_{k_m},\w} \right)^2 \mathbf{1}_{ | \bar{T}^{(d_{k_m})}_{\nu_i}-\mu_{i,d_{k_m},\w}  | > \e \sqrt{v_{k_m,\w}}}  \right]\mathbf{1}_{M_i > d_{k_m}^{(1-\e')/s}} = 0 . \label{Mlarge}
\end{align}
To get an upper bound for \eqref{Mlarge}, first note that our
choice of the subsequence $n_{k_m}$ gives that for $m$ large enough $v_{k_m,\w} \geq \frac{1}{2} \sum_{i=\a_m+1}^{\b_m} \mu_{i,d_{k_m},\w}^2 \geq \frac{a_{k_m}}{2} \mu_{i,d_{k_m},\w}$ for any $i\in
(\a_m, \b_m]$. Thus, for $m$ large enough we can replace the
indicators inside the expectations in \eqref{Mlarge} by the
indicators of the events $\left\{ \bar{T}_{\nu_i}^{(d_{k_m})}  >
(1+\e \sqrt{a_{k_m}/2}) \mu_{i,d_{k_m},\w} \right\}$. Thus, for
$m$ large enough and $i\in(\a_m, \b_m]$, we have
\begin{align*}
& E_\w^{\nu_{i-1}} \left[ \left(
\bar{T}^{(d_{k_m})}_{\nu_i}-\mu_{i,d_{k_m},\w} \right)^2
\mathbf{1}_{ |
\bar{T}^{(d_{k_m})}_{\nu_i}-\mu_{i,d_{k_m},\w} | >
\e \sqrt{v_{k_m,\w}}}  \right] \\
&\qquad \leq E_\w^{\nu_{i-1}} \left[ \left(
\bar{T}^{(d_{k_m})}_{\nu_i}-\mu_{i,d_{k_m},\w} \right)^2
\mathbf{1}_{ \bar{T}^{(d_{k_m})}_{\nu_i}> (1+\e
\sqrt{a_{k_m}/2})
\mu_{i,d_{k_m},\w} }  \right] \\
&\qquad = \int_{1+\e\sqrt{a_{k_m}/2}}^{\infty} P_\w^{\nu_{i-1}}
\left( \bar{T}_{\nu_i}^{(d_{k_m})} > x \mu_{i,d_{k_m},\w}
\right) 2(x-1)\mu_{i,d_{k_m},\w}^2 \, dx
\,.\end{align*}
We want to get an upper bound on the probabilities inside the integral. If $\e'<\frac{1}{3}$ we can use \cite[Lemma 5.9]{pzSL1} to get that for $k$ large enough, $E_\w^{\nu_{i-1}}\left( \bar{T}^{(d_{k})}_{\nu_i} \right)^j \leq 2^j j! \mu_{i,d_k,\w}^j$ for all $n_{k-1}<i\leq n_k$ such that $M_i > d_k^{(1-\e')/s}$.
Multiplying by $(4\mu_{i,d_k,\w})^{-j}$ and summing over $j$ gives
that $E_\w^{\nu_{i-1}} e^{ \bar{T}^{(d_k)}_{\nu_i} /(4
\mu_{i,d_k,\w}) } \leq 2$. Therefore, Chebychev's inequality gives
\[
P_\w^{\nu_{i-1}} \left( \bar{T}_{\nu_i}^{(d_k)} > x \mu_{i,d_k,\w}
\right) \leq e^{- x/4} E_\w^{\nu_{i-1}} e^{
\bar{T}^{(d_k)}_{\nu_i} / (4 \mu_{i,d_k,\w})} \leq 2 e^{-x/4}
\,.\]
Thus, for all $m$ large enough we have for all $\a_m<i\leq \b_m
\leq n_{k_m}$ with $M_i > d_{k_m}^{(1-\e')/s}$ that
\begin{align*}
\int_{1+\e\sqrt{a_{k_m}/2}}^{\infty} P_\w^{\nu_{i-1}} \left(
\bar{T}_{\nu_i}^{(d_{k_m})} > x \mu_{i,d_{k_m},\w} \right)
2(x-1)\mu_{i,d_{k_m},\w}^2 dx 
&\leq \mu_{i,d_{k_m},\w}^2 \int_{1+\e\sqrt{a_{k_m}/2}}^{\infty} 4(x-1)e^{-x/4}
 dx\\
&= \mu_{i,d_{k_m},\w}^2 \: o\!\left( e^{-a_{k_m}^{1/4}} \right) \,.
\end{align*}
Therefore we have that as $m\ra\infty$,
\eqref{Mlarge} is bounded above by
\begin{align}
\lim_{m\ra\infty} o\left( e^{-a_{k_m}^{1/4}} \right) \frac{1}{v_{k_m,\w}} \left( \sum_{i=\a_m+1}^{\b_m} \mu_{i,d_{k_m},\w}^2 \mathbf{1}_{M_i > d_{k_m}^{(1-\e')/s}} \right)
\,. \label{finalest}
\end{align}
However, since 
\begin{align*}
\frac{1}{v_{k_m,\w}} \sum_{i=\a_m+1}^{\b_m} \mu_{i,d_{k_m},\w}^2 &\leq \frac{1}{ \sum_{i=\a_m+1}^{\b_m} \s_{i,d_{k_m},\w}^2 } \left( \sum_{i=\a_m+1}^{\b_m} \s_{i,d_{k_m},\w}^2
 + o\left( d_{k_m}^{2/s} \right) \right)\\
& \leq 1 + \frac{o\left( d_{k_m}^{2/s} \right)}{2 a_{k_m} d_{k_m}^{2/s} + o\left( d_{k_m}^{2/s} \right)},
\end{align*}
we have that \eqref{finalest} tends to zero as $m\ra\infty$. 
This finishes the proof of \eqref{LF2} and thus of Theorem \ref{Tngaussian}.
\end{proof}
\begin{proof}[\textbf{Proof of Theorem \ref{qCLT}:}] \ \\
Choose $\eta \in (0,1)$ such that $\eta < \frac{1}{\bar\nu}$ where $\bar\nu = E_P \nu$, and then choose $n_{k_m}$ as in Theorem \ref{Tngaussian}. Then for $\b_m$ and $\gamma_m$ defined as in \eqref{abgdef}, we have that \eqref{nuLLN} and the fact that $d_k\sim n_k$ give 
\[
\lim_{m\ra\infty} \frac{\nu_{\b_m}}{n_{k_m}} = \eta \bar{\nu} < 1 < \bar\nu = \lim_{m\ra\infty} \frac{\nu_{\gamma_m}}{n_{k_m}}. 
\]
Thus $x_m \sim n_{k_m} \Ra x_m \in [\nu_{\b_m}, \nu_{\gamma_m} ]$ for all $m$ large enough. Therefore, the conditions of Proposition \ref{generalprop} are satisfied with $F(x)= \Phi(x)$. 
\end{proof}
\end{section}

\begin{section}{Quenched Exponential Limits} \label{exponential}
\begin{subsection}{Analysis of $T_\nu$ when $M_1$ is Large} \label{Laplace}
The goal of this subsection is to analyze the quenched distribution of $\bar{T}^{(n)}_\nu$ on ``large'' blocks (i.e. when $M_1>n^{(1-\e)/s}$). We want to show that conditioned on $M_1$ being large, $\bar{T}^{(n)}_\nu / E_\w \bar{T}^{(n)}_\nu$ is approximately exponentially distributed. We do this by showing that the quenched Laplace transform $E_\w \exp \left\{-\l \frac{\bar{T}^{(n)}_\nu}{E_\w \bar{T}^{(n)}_\nu}\right\}$ is approximately $\frac{1}{1+\l}$ on such blocks. 

As was done in \cite{ESZ}, we analyze the quenched Laplace transform of $\bar{T}^{(n)}_\nu$ by decomposing $\bar{T}^{(n)}_\nu$
into a series of excursions away from 0. An excursion is a ``failure'' if the random walk returns to zero before hitting $\nu$ (i.e. if $T_\nu > T_0^+:= \min\{ k > 0: X_k = 0 \}$), and a ``success'' if the random walk reaches $\nu$ before returning to zero (note that classifying an excursion as a failure/sucess is independent of any modifications to the environment left of zero since if the random walk ventures to the left at all, it must be in a failure excursion). Define $p_\w:=P_\w ( T_\nu<T^+_0)$, and let $N$ be a geometric random variable with parameter $p_\w$ (i.e. $P(N=k) = p_\w (1-p_\w)^k$ for $k\in \N$). Also, let $\{F_i\}_{i=1}^\infty$ be an i.i.d. sequence (also independent of $N$) with $F_1$ having the same distribution as $\bar{T}_\nu^{(n)}$ conditioned on $\left\{ \bar{T}^{(n)}_\nu > T_0^+ \right\}$, and let $S$ be a random variable with the same distribution as $T_\nu$ conditioned on $\left\{ T_\nu < T_0^+ \right\}$ and independent of everything else (note that for sucess excursions we can ignore added reflections to the left of zero). Thus, we have that 
\begin{equation}
\bar{T}^{(n)}_\nu \stackrel{Law}{=} S + \sum_{i=1}^N F_i \qquad\text{(quenched).} \label{Tdec}
\end{equation}
In a slight abuse of notation we will still use $P_\w$ for the probabilities of $F_i, S,$ and $N$ to emphasize that their distributions are dependent on $\w$. 
The following results are easy to verify:
\begin{equation}
E_\w N = \frac{1-p_\w}{p_\w}  \quad\text{and}\quad E_\w \bar{T}^{(n)}_\nu = E_\w S + (E_\w N)( E_\w F_1), \label{ETdec}
\end{equation}
\begin{align}
Var_\w \bar{T}^{(n)}_\nu &= (E_\w N)( Var_\w F_1) + (E_\w F)^2 (Var_\w N) + Var_\w S \nonumber \\
&= (E_\w N )(E_\w F^2) + (E_\w F)^2 (Var_\w N - E_\w N) + Var_\w S \nonumber \\
&= (E_\w N )(E_\w F^2) + (E_\w F)^2(E_\w N)^2 + Var_\w S, \label{VTdec}
\end{align}
and 
\begin{align*}
E_\w e^{-\l \bar{T}^{(n)}_\nu} = E_\w e^{-\l S} E_\w\left[ \left(E_\w e^{-\l F_1}\right)^N \right] = E_\w e^{-\l S} \frac{p_\w}{1-(1-p_\w)\left(E_\w e^{-\l F_1}\right)}, \quad \forall \l\geq 0.
\end{align*}
Also, since $e^{-x} \geq 1-x$ for any $x\in\R$ we have for any $\l\geq 0$ that
\begin{align*}
E_\w e^{-\l \bar{T}^{(n)}_\nu} &\geq \left(1-\l E_\w S \right) \frac{p_\w}{1-(1-p_\w)\left(1- \l E_\w F_1\right)} 
=  \frac{1-\l E_\w S }{1 + \l(E_\w N)(E_\w F_1)}
\geq  \frac{1-\l E_\w S}{1 + \l E_\w \bar{T}^{(n)}_\nu},
\end{align*}
where the first equality and the last inequality are from the formulas for $E_\w N$ and $E_\w \bar{T}^{(n)}_\nu$ given in \eqref{ETdec}. Similarly, since $e^{-x} \leq 1-x+\frac{x^2}{2}$ for all $x\geq 0$ we have that for any $\l \geq 0$ that
\begin{align*}
E_\w e^{-\l \bar{T}^{(n)}_\nu} &\leq \frac{p_\w}{1-(1-p_\w)\left(1- \l E_\w F_1 + \frac{\l^2}{2} E_\w F_1^2\right)} \\
&= \frac{1}{1+\l(E_\w N)(E_\w F_1) - \frac{\l^2}{2} (E_\w N)( E_\w F_1^2)}\\
&= \frac{1}{1+\l(E_\w N)(E_\w F_1) - \frac{\l^2}{2} (Var_\w \bar{T}^{(n)}_\nu - (E_\w N)^2(E_\w F_1)^2 - Var_\w S)}\\
&\leq \frac{1}{1+\l(E_\w \bar{T}^{(n)}_\nu - E_\w S) - \frac{\l^2}{2} (Var_\w \bar{T}^{(n)}_\nu - (E_\w \bar{T}^{(n)}_\nu - E_\w S)^2 )},
\end{align*}
where the first equality and last inequality are from \eqref{ETdec} and the second equality is from \eqref{VTdec}. Therefore, replacing $\l$ by $\l/(E_\w \bar{T}^{(n)}_\nu)$ we get 
\begin{equation}
 E_\w e^{-\l \frac{\bar{T}^{(n)}_\nu}{E_\w \bar{T}^{(n)}_\nu}} \geq \left(1-\l \frac{E_\w S}{E_\w \bar{T}^{(n)}_\nu} \right) \frac{1}{1 + \l } \, , \label{mgflb}
\end{equation}
and
\begin{align}
E_\w e^{-\l \frac{\bar{T}^{(n)}_\nu}{E_\w \bar{T}^{(n)}_\nu}}
&\leq \frac{1}{1+\l - \l \frac{E_\w S}{E_\w \bar{T}^{(n)}_\nu } - \frac{\l^2}{2} \left(\frac{Var_\w \bar{T}^{(n)}_\nu}{(E_\w \bar{T}^{(n)}_\nu)^2 } -\frac{(E_\w \bar{T}^{(n)}_\nu - E_\w S)^2}{(E_\w \bar{T}^{(n)}_\nu)^2 } \right)} \nonumber \\
&\leq \frac{1}{1+\l - (\l+\l^2) \frac{E_\w S}{E_\w \bar{T}^{(n)}_\nu } - \frac{\l^2}{2} \left(\frac{Var_\w \bar{T}^{(n)}_\nu}{(E_\w \bar{T}^{(n)}_\nu)^2 } - 1 \right)} \, . \label{mgfub}
\end{align}
Therefore, we have reduced the problem of showing $E_\w e^{-\l \frac{\bar{T}^{(n)}_\nu}{E_\w \bar{T}^{(n)}_\nu}} \approx \frac{1}{1+\l}$ when $M_1$ is large to showing that $\frac{E_\w S}{E_\w \bar{T}^{(n)}_\nu } \approx 0$ and $\frac{Var_\w \bar{T}^{(n)}_\nu}{(E_\w \bar{T}^{(n)}_\nu)^2 } \approx 1$ when $M_1$ is large. 
In order to analyze $E_\w S$, we define a modified environment which is essentially the environment the random walker ``sees'' once it is told that it reaches $\nu$ before returning to zero. A simple computation similar to the one in \cite[Remark 2 on pages 222-223]{zRWRE} gives that the random walk conditioned to reach $\nu$ before returning to zero is a homogeneous markov chain with transition probabilities given by $\bw_i := P_\w^i(X_1=i+1 | T_\nu < T^+_0 )$. 
Then the definition of $\bw_i$ gives that $\bw_0=\bw_1 =1$, and for $i \in [2,\nu)$ we have $\bw_i = \frac{\w_i P_\w^{i+1}( T_\nu < T_0 )}{P_\w^i ( T_\nu < T_0)}$. Using the hitting time formulas in \cite[(2.1.4)]{zRWRE} we have 
\begin{equation}
\bw_i = \frac{\w_i R_{0,i}}{R_{0,i-1}} \quad \forall i\in[2,\nu), \quad\text{where}\quad R_{0,i}:= \sum_{j=0}^i \Pi_{0,j}. \label{Rdef}
\end{equation}
Let $\bar\rho_i:=\frac{1-\bw_i}{\bw_i}$ and define $\bar{\Pi}_{i,j},$ and $\bar{W}_{i,j}$ analogously to $\Pi_{i,j}$ and $W_{i,j}$ 
using $\bar\rho_i$ in place of $\rho_i$. Then the above formulas for $\bw_i$ give that $\bar\rho_0=\bar\rho_1=0$ and $\bar\rho_i = \rho_i \frac{R_{0,i-2}}{R_{0,i}} \quad\forall i\in[2,\nu)$. Thus,
\begin{equation}
\bar{\Pi}_{i,j} = \Pi_{i,j} \frac{R_{0,i-2}R_{0,i-1}}{R_{0,j-1}R_{0,j}},\quad\forall 2\leq i\leq j<\nu. \label{rhobarformulas}
\end{equation}
Note that since $R_{0,i}\leq R_{0,j}$ for any $0\leq i\leq j$ we have from \eqref{rhobarformulas} that 
\begin{equation}
\bar\Pi_{i,j}\leq \Pi_{i,j} \quad \text{for any } 0\leq i\leq j<\nu \label{PibarlessPi}
\end{equation}
Now, since $E_\w S = E_{\bw} T_\nu$ we get from \eqref{qvformula} that $E_\w S = \nu + 2\sum_{j=2}^{\nu-1} \bar{W}_{2,j} = \nu + 2\sum_{j=2}^{\nu-1} \sum_{i=2}^j \bar{\Pi}_{i,j}$.
Therefore, letting $\bar{M}_1:= \max \{ \bar{\Pi}_{i,j}: 0 \leq i\leq j<\nu \}$ we get the bound 
\begin{equation}
E_\w S \leq \nu + 2\nu^2 \bar{M}_1. \label{ESbound}
\end{equation}
Thus, we need to get bounds on the tail of $\bar{M}_1$. To this end, recall the definition of $M_1$ in \eqref{Mdef} and define $\tau := \max \{ k \in [1,\nu]: \Pi_{0,k-1} = M_1 \}$. Then, define
\begin{equation}
M^-:= \min \{ \Pi_{i,j}: 0<i\leq j < \tau \} \wedge 1, \quad\text{and}\quad M^+:= \max \{ \Pi_{i,j} : \tau < i \leq j < \nu  \}\vee 1 \, .\label{Ddef}
\end{equation} 
\begin{lem} \label{Mpmtail}
For any $\e,\d>0$ we have
\begin{equation}
P( M^- < n^{-\d} , M_1 > n^{(1-\e)/s} ) = o(n^{-1+\e-\d s + \e'}), \quad \forall \e' > 0, \label{xdown}
\end{equation}
and
\begin{equation}
P( M^+ > n^{\d} , M_1 > n^{(1-\e)/s} ) = o(n^{-1+\e-\d s + \e'}), \quad \forall \e' > 0, \label{xup}
\end{equation}
\end{lem}
\begin{proof}
Since $\Pi_{0,\tau-1} = M_1$ by definition we have
\begin{align}
P( M^- < n^{-\d} , M_1 > n^{(1-\e)/s} ) &\leq P\left( \exists 0 < i\leq j < \tau - 1: \Pi_{i,j} < n^{-\d}, \quad \Pi_{0,\tau-1} > n^{(1-\e)/s}   \right) \nonumber \\
&\leq P(\tau > b_n) + \sum_{ 0<i\leq j < k < b_n }P\left( \Pi_{i,j} < n^{-\d}, \quad \Pi_{0,k} > n^{(1-\e)/s}   \right) \nonumber \\
&\leq P(\nu > b_n) + \sum_{ 0<i\leq j < k < b_n }P\left( \Pi_{0,i-1}\Pi_{j+1,k} > n^{(1-\e)/s + \d}   \right) \label{xfluc}.
\end{align}
Since \eqref{nutail} gives that $P(\nu > b_n) \leq C_1 e^{-C_2 b_n}$ we need only handle the second term in \eqref{xfluc} to prove \eqref{xdown}.
However, Chebychev's inequality and the fact that $P$ is a product measure give that
\[
P\left( \Pi_{0,i-1}\Pi_{j+1,k} > n^{(1-\e)/s + \d}   \right) \leq n^{-1+\e-\d s} (E_P \rho^s)^{i+k-j} = n^{-1+\e-\d s}.
\]
Since the number of terms in the sum in \eqref{xfluc} is at most $(b_n)^3 = o(n^{\e'})$ we have proved \eqref{xdown}. The proof of \eqref{xup} is similar:
\begin{align*}
P( M^+ > n^{\d} , M_1 > n^{(1-\e)/s} ) &\leq P\left( \exists \tau < i\leq j < \nu: \Pi_{i,j} > n^{\d}, \quad \Pi_{0,\tau-1} > n^{(1-\e)/s}   \right) \nonumber \\
&\leq P(\nu > b_n) + \sum_{ 0\leq k <i\leq j < b_n }P\left( \Pi_{0,k}\Pi_{i,j} > n^{(1-\e)/s + \d}   \right)\\
&\leq C_1 e^{-C_2 b_n} + (b_n)^3 n^{-1+\e-\d s} = o(n^{-1+\e-\d s + \e'})
\end{align*}
\end{proof}
\begin{cor} \label{Sbig}
For any $\e,\d>0$ we have
\[
P\left( E_\w S \geq n^{5\d} , M_1 > n^{(1-\e)/s} \right) = o(n^{-1+\e-\d s + \e'}), \quad \forall \e'>0. 
\]
\end{cor}
\begin{proof}
Recall that \eqref{ESbound} gives $E_\w S \leq \nu + 2 \nu^2 \bar{M}_1$.
We will use $M^-$ and $M^+$ to get bounds on $\bar{M}_1$. First, note that for any $i\in[0,\tau)$ we have
\[
R_{0,i} = \sum_{k=0}^i \Pi_{0,k} = \Pi_{0,i} + \sum_{k=0}^{i-1} \frac{\Pi_{0,i}}{\Pi_{k+1,i}} \leq \Pi_{0,i}\left(\frac{i+1}{M^-}\right).
\]
Note also that $R_{0,j} \geq \Pi_{0,j}$ holds for any $j\geq 0$. Thus, for any $2\leq i \leq  j \leq \tau$  we have
\begin{align*}
\bar{\Pi}_{i,j} &= \Pi_{i,j} \frac{R_{0,i-2}R_{0,i-1}}{R_{0,j-1}R_{0,j}} 
\leq \Pi_{i,j} \left(\frac{i}{M^-}\right)^2 	\frac{\Pi_{0,i-2}\Pi_{0,i-1}}{\Pi_{0,j-1}\Pi_{0,j}} 
=\left(\frac{i}{M^-}\right)^2 	\frac{1}{\Pi_{i-1,j-1}} \leq \frac{i^2}{(M^-)^3}.
\end{align*}
Also, from \eqref{PibarlessPi} we have that $\bar{\Pi}_{i,j} \leq \Pi_{i,j} \leq M^+$ for $\tau < i\leq j<\nu$. 
Therefore we have that $\bar{M}_1 \leq \frac{\nu^2 M^+}{(M^-)^3}$ (note that here we used that $M^-\leq 1$ and $M^+ \geq 1$). Thus,
\[
P\left( E_\w S \geq n^{5\d} , M_1 > n^{(1-\e)/s} \right) \leq P\left( \nu + \frac{2 \nu^4 M^+}{(M^-)^3} > n^{5\d} , M_1 > n^{(1-\e)/s} \right).
\]
An easy argument using \eqref{nutail} and Lemma \ref{Mpmtail} finishes the proof. 
\end{proof}
\begin{lem} \label{VarET2compare}
For any $\e,\d>0$ we have
\[
Q\left( \left| \frac{Var_\w \bar{T}^{(n)}_\nu}{(E_\w \bar{T}^{(n)}_\nu)^2} - 1 \right| \geq  n^{-\d} ,\quad  M_1 > n^{(1-\e)/s} \right) = o(n^{-2+2\e+\d s+\e'}), \quad \forall \e'>0
\]
\end{lem}
\begin{proof}
Recall that from \cite[(61)]{pzSL1} we have that there exist explicit non-negative random variables $D^+(\w)$ and $D^-(\w)$ such that
\[
\left( E_\w \bar{T}^{(n)}_\nu \right)^2 -D^+(\w) \leq Var_\w \bar{T}^{(n)}_\nu  \leq
\left( E_\w \bar{T}^{(n)}_\nu \right)^2 + 8R_{0,\nu-1}D^-(\w), 
\]
where $R_{0,\nu-1}$ is defined as in \eqref{Rdef}.
Therefore, since $E_\w \bar{T}^{(n)}_\nu \geq M_1$, we have
\begin{align}
&Q\left( \left| \frac{Var_\w \bar{T}^{(n)}_\nu}{(E_\w \bar{T}^{(n)}_\nu)^2} - 1 \right| \geq  n^{-\d} ,  M_1 > n^{(1-\e)/s} \right) \nonumber \\
&\qquad \leq Q\left( 8 R_{0,\nu-1} D^-(\w) > n^{(2-2\e)/s - \d} \right) + Q\left( D^+(\w) >n^{(2-2\e)/s-\d} \right). \label{Dpm}
\end{align}
However, \cite[Lemma 5.2 \& Corollary 5.4]{pzSL1} give respectively that $Q(D^+(\w) > x) = o(x^{-s+\e''})$ and $Q\left( R_{0,\nu-1} D^-(\w) > x \right) = o(x^{-s+\e''})$ for any $\e''>0$. Therefore, both terms in \eqref{Dpm} are of order
$o\left( n^{-2+2\e +\d s + \e''((2-2\e)/s - \d)} \right)$. The lemma then follows since $\e''>0$ is arbitrary. 
\end{proof}

For any $i$, define the scaled quenched Laplace transforms 
$\phi_{i,n}(\l) := E_\w^{\nu_{i-1}} \exp\left\{-\l \frac{\bar{T}^{(n)}_{\nu_i}}{\mu_{i,n,\w}} \right\}$.
\begin{lem} \label{mgflem}
Let $\e<\frac{1}{8}$, and define $\e':= \frac{1-8\e}{5} > 0$. Then
\[
Q\left( \exists \l \geq 0: \phi_{1,n}(\l) \notin \left[ \frac{1-\l n^{-\e /s}}{1+\l} , \frac{1}{1+\l-\left(\l+\frac{3 \l^2}{2}\right)n^{-\e/s} } \right],\: M_1 > n^{(1-\e)/s} \right) = o\left( n^{-1-\e'} \right). 
\]
\end{lem}
\begin{proof}
Recall from \eqref{mgflb} and \eqref{mgflb} that 
\[
\left(1-\l \frac{E_\w S}{E_\w \bar{T}^{(n)}_\nu} \right) \frac{1}{1 + \l }  \leq \phi_{1,n}(\l) \leq \frac{1}{1+\l - (\l+\l^2) \frac{E_\w S}{E_\w \bar{T}^{(n)}_\nu } - \frac{\l^2}{2} \left(\frac{Var_\w \bar{T}^{(n)}_\nu}{(E_\w \bar{T}^{(n)}_\nu)^2 } - 1 \right)}
\]
for all $\l\geq 0$. Therefore
\begin{align*}
&Q\left( \exists \l \geq 0: \phi_{1,n}(\l) \notin \left[ \frac{1-\l n^{-\e /s}}{1+\l} , \frac{1}{1+\l-\left(\l+3 \l^2 /2\right)n^{-\e/s} } \right],\: M_1 > n^{(1-\e)/s} \right) \\
&\qquad \leq Q\left( \frac{E_\w S}{E_\w \bar{T}^{(n)}_\nu} \geq n^{-\e/s}, \quad M_1 \geq n^{(1-\e)/s}  \right) + Q\left( \frac{Var_\w \bar{T}^{(n)}_\nu }{(E_\w \bar{T}^{(n)}_\nu)^2} - 1 \geq n^{-\e/s} , \quad M_1 \geq n^{(1-\e)/s} \right)
\end{align*}
Now, since $E_\w \bar{T}^{(n)}_\nu \geq M_1$ we have
\[
Q\left( \frac{E_\w S}{E_\w \bar{T}^{(n)}_\nu} \geq n^{-\e/s}, \: M_1 \geq n^{(1-\e)/s}  \right)  \leq Q\left( E_\w S \geq n^{(1-2\e)/s}, \: M_1 \geq n^{(1-\e)/s}  \right) = o\left( n^{-(6-8\e)/5} \right),
\]
where the last equality is from Corollary \ref{Sbig}. Also, by Lemma \ref{VarET2compare} we have 
\[
Q\left( \frac{Var_\w \bar{T}^{(n)}_\nu }{(E_\w \bar{T}^{(n)}_\nu)^2} - 1 \geq n^{-\e/s} , \quad M_1 \geq n^{(1-\e)/s} \right) = o\left( n^{-2+4\e} \right).
\]
Then, since $-2+4\e < \frac{-6+8\e}{5}$ when $\e<\frac{1}{8}$ the lemma is proved. 
\end{proof} 
\begin{cor} \label{explimit}
Let $\e<\frac{1}{8}$. Then $P-a.s.$, for any sequence $i_k=i_k(\w)$ such that $i_k\in (n_{k-1},n_k]$ and $M_{i_k} > d_k^{(1-\e)/s}$ we have
\begin{equation}
\lim_{k\ra\infty} \phi_{i_k,d_k}(\l) = \frac{1}{1+\l},\quad \forall \l\geq 0, \label{mgfbounds}
\end{equation}
and thus
\begin{equation}
\lim_{k\ra\infty} P_\w^{\nu_{i_k-1}} \left( \bar{T}^{(d_k)}_{\nu_{i_k}} > x \mu_{i_k,d_k,\w} \right) =  \Psi(x), \quad \forall x\in \R. \label{expl}
\end{equation}
\end{cor}
\begin{proof}
For $i\in(n_{k-1},n_k]$ and all $k$ large enough $\phi_{i,d_k}(\l)$ only depends on the environment to the right of zero, and thus has the same distribution under $P$ and $Q$. Therefore, Lemma \ref{mgflem} gives that there exists an $\e'>0$ such that 
\begin{align*}
&P\left( \exists i\in (n_{k-1},n_k], \l\geq 0: \phi_{i,d_k}(\l) \notin \left[ \frac{1-\l d_k^{-\e /s}}{1+\l} , \frac{1}{1+\l-\left(\l+\frac{3 \l^2}{2}\right)d_k^{-\e/s} } \right],\: M_i > d_k^{(1-\e)/s} \right) \\
&\quad \leq d_k Q\left( \exists \l\geq 0: \phi_{1,d_k}(\l) \notin \left[ \frac{1-\l d_k^{-\e /s}}{1+\l} , \frac{1}{1+\l-\left(\l+\frac{3 \l^2}{2}\right)d_k^{-\e/s} } \right],\: M_1 > d_k^{(1-\e)/s} \right) \\
&\quad = o\left( d_k^{-\e'} \right).
\end{align*}
Since this last term is summable in $k$, the Borel-Cantelli Lemma gives that $P-a.s.$ there exists a $k_0=k_0(\w)$ such that for all $k\geq k_0$ we have
\[
i\in (n_{k-1},n_k] \text{ and } M_i\geq d_k^{(1-\e)/s} \Ra \phi_{i,d_k}(\l) \in \left[ \frac{1-\l d_k^{-\e /s}}{1+\l} , \frac{1}{1+\l-\left(\l+\frac{3 \l^2}{2}\right)d_k^{-\e/s} } \right] \quad \forall \l\geq 0,
\]
which proves \eqref{mgfbounds}. Then, \eqref{expl} follows immediately because $\frac{1}{1+\l}$ is the Laplace transform of an exponential disribution. 
\end{proof}
\end{subsection}

\begin{subsection}{Quenched Exponential Limits Along a Subsequence}
In the previous subsection we showed that the time to cross a single large block is approximately exponential. In this section we show that there are subsequences in the environment where the crossing time of a single block dominates the crossing times of all the other blocks. In this case the crossing time of all the blocks is approximately exponentially distributed. 
Recall the definition of $M_i$ in \eqref{Mdef}. For any integer $n\geq 1$, and constants $C>1$ and $\eta>0$, define the event
\[
\mathcal{D}_{n,C,\eta}:= \left\{ \exists i\in\left[1,\eta n \right]: M_i^2 \geq C \sum_{j: i\neq j\leq n} \s_{j,n,\w}^2  \right\}
\]
\begin{lem} \label{onebigblock}
Assume $s<2$. Then for any $C>1$ and $\eta>0$ we have $\liminf_{n\ra\infty} Q\left( \mathcal{D}_{n,C,\eta} \right) > 0$. 
\end{lem}
\begin{proof}
First, note that since $\s_{i,n,\w}^2 \geq M_i^2$ and $C>1$ we have 
\begin{equation}
Q\left(\mathcal{D}_{n,C,\eta}\right) = \sum_{i=1}^{\eta n} Q\left( M_i^2 \geq C \sum_{j: i\neq j\leq n} \s_{j,n,\w}^2 \right). \label{disjoint}
\end{equation}
Thus, we want to get a lower bound on $Q\left( M_i^2 \geq C \sum_{j:i\neq j\leq n} \s_{j,n,\w}^2 \right)$ that is uniform in $i$. 
The following formula for the quenched variance of $\bar{T}^{(n)}_\nu$ can be deduced from \eqref{qvformula} by setting $\rho_{\nu_{-b_n}=0}$:
\begin{align*}
Var_\w \bar{T}^{(n)}_\nu &= 4\sum_{j=0}^{\nu-1}(W_{\nu_{-b_n}+1,j}+W_{\nu_{-b_n}+1,j}^2) + 8\sum_{j=0}^{\nu-1}\sum_{i=\nu_{-b_n}+1}^j \Pi_{i+1,j}(W_{\nu_{-b_n}+1,i}+W_{\nu_{-b_n}+1,i}^2)\\
&\leq 4\sum_{j=0}^{\nu-1}(W_{\nu_{-b_n}+1,j}+W_{\nu_{-b_n}+1,j}^2) + 8\sum_{j=0}^{\nu-1}\sum_{i=\nu_{-b_n}+1}^j W_{\nu_{-b_n}+1,j}(1+W_{\nu_{-b_n}+1,i})\\
&\leq 4\sum_{j=0}^{\nu-1}(W_{\nu_{-b_n}+1,j}+W_{\nu_{-b_n}+1,j}^2) + 8\left( \sum_{j=0}^{\nu-1} W_{\nu_{-b_n}+1,j} \right)\left( \sum_{i=\nu_{-b_n}+1}^{\nu-1} (1+W_{\nu_{-b_n}+1,i}) \right)  ,
\end{align*}
where the first inequality is because $W_{\nu_{-b_n}+1,j} = W_{i+1,j}+\Pi_{i+1,j}W_{\nu_{-b_n}+1,i}$. 
Next, note that if $\nu_{k-1}\leq j < \nu_{k}$ for some $k>-b_n$, then 
\[
W_{\nu_{-b_n}+1,j} = \sum_{l=\nu_{-b_n}+1}^j \Pi_{l,j} = \sum_{l=\nu_{-b_n}+1}^{\nu_{k-1}-1} \Pi_{l,\nu_{k-1}-1}\Pi_{\nu_{k-1},j} + \sum_{l=\nu_{k-1}}^j \Pi_{l,j} \leq (\nu_k-\nu_{-b_n}) M_k,
\]
where the last inequality is because, under $Q$, $\Pi_{l,\nu_{k-1}-1} < 1$ for all $l<\nu_{k-1}$.
Therefore,
\begin{align*}
Var_\w \bar{T}^{(n)}_\nu &\leq 4 \nu_1 \left( (\nu_1-\nu_{-b_n})M_1 +  (\nu_1-\nu_{-b_n})^2M_1^2 \right) \\
&\qquad + 8 \left(\nu_1 (\nu_1-\nu_{-b_n})M_1 \right)
\left( (\nu_1-\nu_{-b_n}) + \sum_{i=-b_n+1}^{1} (\nu_k-\nu_{k-1})(\nu_k-\nu_{-b_n}) M_k \right)\\
&\leq \left( \nu_1-\nu_{-b_n} \right)^4 \left( 12 M_1 + 4 M_1^2 + 8 M_1 \sum_{k=-b_n+1}^1 M_k \right). 
\end{align*}
Similarly, we have that 
$\s_{j,n,\w}^2 \leq \left( \nu_j - \nu_{j-1-b_n} \right)^4 \left( 12 M_j + 4  M_j^2 + 8 M_j \sum_{k=j-b_n}^j M_k \right)$ $Q-a.s.$ for any $j$. 
Now, define the events
\begin{equation}
F_n:= \bigcap_{j\in(-b_n,n]} \left\{ \nu_j-\nu_{j-1} \leq b_n \right\}, \quad\text{and}\quad G_{i,n,\e}:= \bigcap_{j\in[i-b_n,i+b_n]\backslash\{i\}} \left\{ M_j \leq n^{(1-\e)/s} \right\} \label{FGdef}
\end{equation}
Then, on the event $F_n\cap G_{i,n,\e} \cap \left\{ M_i \leq 2 n^{1/s} \right\}$ we have for $j\in (i,i+b_n]$ that
\begin{align*}
\s_{j,n,\w}^2 &\leq b_n^4(b_n+1)^4 \left( 12 n^{(1-\e)/s} + 4 n^{(2-2\e)/s} + 8 n^{(1-\e)/s}(b_n n^{(1-\e)/s} + 2 n^{1/s})\right)\\
&\leq  b_n^5(b_n+1)^4 \left( 12 n^{(1-\e)/s} + 12 n^{(2-2\e)/s} + 16 n^{(2-\e)/s} \right) \leq 80 b_n^9 n^{(2-\e)/s},
\end{align*}
where the last inequality holds for all $n$ large enough.
Therefore, for all $n$ large enough
\begin{align*}
& Q\left( M_i^2 \geq C \sum_{j:i\neq j\leq n} \!\!\!\! \s_{j,n,\w}^2 \right) \\
&\qquad \geq Q\left(4 n^{2/s} \geq  M_i^2 \geq C \sum_{j:i\neq j\leq n} \s_{j,n,\w}^2 , \ F_n, \ G_{i,n,\e} \right)\\
&\qquad \geq Q\left( 4 n^{2/s} \geq M_i^2 \geq C \left( \sum_{j\in[1,n]\backslash [i,i+b_n]} \!\!\!\!\!\! \s_{j,n,\w}^2 + 80 b_n^9 n^{(2-\e)/s} \right) , \ F_n, \ G_{i,n,\e} \right)\\
&\qquad \geq Q\left( M_i \in[n^{1/s},2n^{1/s}] , \  \nu_i-\nu_{i-1} \leq b_n \right) \\
&\qquad \qquad \times Q\left(  \sum_{j\in[1,n]\backslash [i,i+b_n]} \!\!\!\!\!\! \s_{j,n,\w}^2 + 80 b_n^9 n^{(2-\e)/s} \leq \frac{n^{2/s}}{C},  \ \tilde{F}_n, \ G_{i,n,\e} \right),
\end{align*}
where $\tilde{F}_n := \{ \nu_j-\nu_{j-1} \leq b_n, \; \forall j\in(-b_n, n]\backslash\{k\} \} \supset F_n$. 
Note that in the last inequality we used that $\s_{j,n,\w}^2$ is independent of $M_i$ for $j\notin[i,i+b_n]$. Also, note that we can replace $\tilde{F}_n$ by $F_n$ in the last line above because it will only make the probability smaller. Then, since $\sum_{j\in[1,n]\backslash [i,i+b_n]} \s_{j,n,\w}^2 \leq Var_\w T_{\nu_n}$ we have
\begin{align}
&Q\left( M_i^2 \geq C \sum_{j:i\neq j\leq n} \s_{j,n,\w}^2 \right) \nonumber \\
&\qquad\geq Q\left( M_1 \in [ n^{1/s}, 2n^{1/s} ] , \  \nu \leq b_n \right)Q\left( Var_\w T_{\nu_n} \leq n^{2/s}C^{-1} -40 b_n^7 n^{(2-\e)/s},  \ F_n, \ G_{i,n,\e} \right) \nonumber \\
&\qquad\geq \left( Q(M_1 \in [ n^{1/s}, 2n^{1/s} ] ) - Q(\nu > b_n)   \right) \nonumber \\
&\qquad\qquad \times \left( Q\left(  Var_\w T_{\nu_n} \leq n^{2/s} ( C^{-1} - 40 b_n^7 n^{-\e/s} )  \right) - Q(F_n^c) - Q(G_{i,n,\w}^c)  \right) \nonumber \\
&\qquad\sim C_3 (1-2^{-s})\frac{1}{n} \, L_{\frac{s}{2}, b}\left(C^{-1}\right), \label{obblb}
\end{align}
where the asymptotics in the last line are from \eqref{nutail}, \eqref{Mtail}, and Theorem \ref{Varstable}, as well as the fact that $Q(F_n^c) + Q(G_{i,n,\w}^c) \leq (n+b_n)Q(\nu> b_n) + 2b_n Q(M_1 > n^{(1-\e)/s}) = \bigo\left(n e^{-C_2 b_n} \right) + o(n^{-1+2\e})$ due to \eqref{nutail} and \eqref{Mtail}. Combining \eqref{disjoint} and \eqref{obblb} finishes the proof.
\end{proof}
\begin{cor} \label{subseq}
Assume $s<2$. Then for any $\eta\in(0,1)$, $P-a.s.$ there exists a subsequence $n_{k_m}= n_{k_m}(\w,\eta)$ of $n_k=2^{2^k}$ such that for $\a_m, \b_m,$ and $\gamma_m$ defined as in \eqref{abgdef}
we have that 
\begin{equation}
\exists i_m = i_m(\w,\eta) \in (\a_m, \b_m]:\quad M_{i_m}^2 \geq m \!\!\!\! \sum_{j\in(\a_m, \gamma_m]\backslash\{i_m\}} \!\!\!\! \s_{j,d_{k_m},\w}^2 \, .\label{sscond}
\end{equation}
\end{cor}
\begin{proof}
Define the events 
\[
\mathcal{D}_{k,C,\eta}':= \left\{ \exists i\in (n_{k-1}, n_{k-1}+\eta d_k]: M_i^2 \geq C \sum_{j\in(n_{k-1},n_k]\backslash\{i\}  } \s_{j,d_k,\w}^2  \right\}.
\]
Note that since $Q$ is invariant under shifts of the $\nu_i$, $Q(\mathcal{D}_{k,C,\eta}') = Q(\mathcal{D}_{d_k,C,\eta})$. Also, due to the reflections of the random walk the event $\mathcal{D}_{k,C,\eta}'$ only depends on the environment between $\nu_{n_{k-1}-b_{d_k}}$ and $\nu_{n_k}$. Thus, for $k$ large enough $\mathcal{D}_{k,C,\eta}'$ only depends on the environment to the right of zero and therefore $P( \mathcal{D}_{k,C,\eta}') =  Q(\mathcal{D}_{k,C,\eta}') = Q(\mathcal{D}_{d_{k},C,\eta})$. Therefore $\liminf_{k\ra\infty} P( \mathcal{D}_{k,C,\eta}') >0$. Also, since $n_{k-1}-b_{d_k} > n_{k-2}$ for all $k\geq 4$, we have that $\{ \mathcal{D}_{2k,C,\eta}' \}_{k=2}^\infty$ is an independent sequence of events. Thus, we get that for any $C>1$ and $\eta\in(0,1)$, infinitely many of the events $\mathcal{D}_{k,C,\eta}$ occur $P-a.s.$ Therefore, $P-a.s.$ there is a subsequence $k_m=k_m(\w)$ such that $\w\in \mathcal{D}_{k_m,m,\eta}$ for all $m$. In particular, for this subsequence $k_m$ we have that \eqref{sscond} holds.
\end{proof}
\begin{thm}\label{Tnexplimit}
Assume $s<2$. Then for any $\eta\in(0,1)$, $P-a.s.$ there exists a subsequence $n_{k_m}=n_{k_m}(\w, \eta)$ of $n_k=2^{2^k}$ such that for $\a_m,\beta_m$ and $\gamma_m$ defined as in \eqref{abgdef} and any sequence $x_m \in \left(\nu_{\beta_m} , \nu_{\gamma_m} \right]$  we have
\[
\lim_{m\ra\infty} P_\w^{\nu_{\a_m}} \left( \frac{\bar{T}_{x_m}^{(d_{k_m})} - E_\w^{\nu_{\a_m}} \bar{T}_{x_m}^{(d_{k_m})} }{\sqrt{v_{k_m,\w}}} \leq x \right) = \Psi(x+1), \quad \forall x\in\R.
\]
\end{thm}
\begin{proof}
First, note that 
\[
 P\left( \max_{j\in (n_{k-1}, n_{k}]} M_j \leq d_k^{(1-\e)/s} \right) = \left( 1- P\left(M_1 > d_k^{(1-\e)/s}\right) \right)^{d_k} = o\left( e^{-d_k^{\e/2}} \right),
\]
where the last equality is due to \eqref{Mtail}. Therefore, the Borel-Cantelli Lemma gives that $P-a.s.$ we have
\begin{equation}
\max_{j\in (n_{k-1}, n_{k}]} M_j > d_k^{(1-\e)/s} \quad \text{ for all } k \text{ large enough.} \label{maxMlarge}
\end{equation}
Therefore, $P-a.s.$ we may assume that \eqref{maxMlarge} holds, the conclusion of Corollary \ref{explimit} holds, and that there exist subsequences $n_{k_m}=n_{k_m}(\w,\eta)$ and $i_m=i_m(\w,\eta)$ as specified in Corollary \ref{subseq}.
Then, by the choice of our subsequence $n_{k_m}$, only the crossing of the largest block (i.e. from $\nu_{i_m-1}$ to $\nu_{i_m}$) is relevant in the limiting distribution. Indeed,
\begin{align*}
&P_\w^{\nu_{\a_m}} \left( \left| \frac{ \left( \bar{T}_{\nu_{i_m-1}}^{(d_{k_m})} - E_\w ^{\nu_{\a_m}}\bar{T}_{\nu_{i_m-1}}^{(d_{k_m})} \right) + \left( \bar{T}_{x_m}^{(d_{k_m})} - \bar{T}_{\nu_{i_m}}^{(d_{k_m})}  - E_\w^{\nu_{i_m}} \bar{T}_{x_m}^{(d_{k_m})} \right) }{\sqrt{v_{k_m,\w}}} \right| \geq \e \right) \\
&\qquad \leq \frac{ Var_\w \left( \bar{T}_{x_m}^{(d_{k_m})} - \bar{T}_{\nu_{\a_m}}^{(d_{k_m})} \right) - \s_{i_m,d_{k_m},\w}^2 } { \e^2 v_{k_m,\w} }  \leq \frac{ \sum_{j\in (\a_m,\gamma_m]\backslash \{i_m\}} \s_{j,d_{k_m},\w}^2 }{ \e^2 M_{i_m}^2 } \leq \frac{1}{\e^2 m},
\end{align*}
where in the second to last inequality we used that $v_{k_m,\w} \geq \s_{i_m,d_{k_m},\w}^2 \geq M_{i_m}^2$, and the last inequality is due to our choice of the sequence $i_m$. 
Thus we have reduced the proof of the Theorem to showing that
\begin{equation}
\lim_{m\ra\infty} P_\w^{\nu_{i_m-1}} \left( \frac{\bar{T}_{\nu_{i_m}}^{(d_{k_m})} - \mu_{i_m,d_{k_m},\w} }{\sqrt{v_{k_m,\w}}} \leq x \right) = \Psi(x+1), \quad \forall x\in\R. 
\end{equation}
Now, since $i_m$ is chosen so that $M_{i_m} = \max_{j\in (n_{k_m-1}, n_{k_m}]} M_j$, we have that $M_{i_m} \geq d_{k_m}^{(1-\e)/s}$ for any $\e>0$ and all $m$ large enough. Then, the conclusion of Corollary \ref{explimit} gives that
\[
\lim_{m\ra\infty} P_\w^{\nu_{i_m-1}} \left( \frac{ \bar{T}^{(d_{k_m})}_{\nu_{i_m}} }{ \mu_{i_m,d_{k_m},\w} } \leq x \right) = \Psi(x).
\]
Thus, the proof will be complete if we can show
\begin{equation}
\lim_{m\ra\infty} \frac{\mu_{i_m,d_{k_m},\w}}{\sqrt{v_{k_m,\w}}} = 1. \label{vovermuim}
\end{equation}
However, by our choice of $n_{k_m}$ and $i_m$ we have 
\[
\s_{i_m,d_{k_m},\w}^2 \geq M_{i_m}^2 \geq  m \sum_{j\in (\a_m, \gamma_m] \backslash\{i_m\} } \s_{j,d_{k_m}\w}^2 = m\left( v_{k_m,\w} - \s_{i_m,d_{k_m},\w}^2 \right),
\]
which implies that
\begin{equation}
1\leq \frac{ v_{k_m,\w} }{ \s_{i_m,d_{k_m},\w}^2 } \leq \frac{m+1}{m} \underset{m\ra\infty}{\longrightarrow} 1. \label{vovers}
\end{equation}
Also, we can use Lemma \ref{VarET2compare} to show that for $k$ large enough and $\e>0$ 
\begin{align*}
&P\left( \exists i\in(n_{k-1},n_k]: \left| \frac{\s_{i,d_k,\w}^2}{\mu_{i,d_k,\w}^2} - 1 \right| \geq d_k^{-\e/s},\quad M_i \geq d_k^{(1-\e)/s} \right) \\
&\quad \leq d_k Q\left( \left| \frac{ Var_\w \bar{T}^{(d_k)}_\nu }{ \left( E_\w \bar{T}^{(d_k)}_\nu \right)^2 } - 1 \right| \geq d_k^{-\e/s},\quad M_1 \geq d_k^{(1-\e)/s} \right) = o\left(d_k^{-1+4\e}\right).
\end{align*}
Then, for $\e<\frac{1}{4}$ the Borel-Cantelli Lemma gives that $P-a.s.$ there exists a $k_0=k_0(\w)$ such that for $k\geq k_0$ and $i\in(n_{k-1}, n_k]$ with $M_i\geq d_k^{(1-\e)/s}$ we have $\left| \frac{\s_{i,d_k,\w}^2}{\mu_{i,d_k,\w}^2} - 1 \right| < d_k^{-\e/s}$. In particular, since $M_{i_m} \geq d_{k_m}^{(1-\e)/s}$ for all $m$ large enough, we have that
\begin{equation}
\lim_{m\ra\infty} \frac{ \s_{i_m,d_{k_m},\w}^2 }{ \mu_{i_m,d_{k_m},\w}^2 } = 1. \label{soverm}
\end{equation}
Since \eqref{vovers} and \eqref{soverm} imply \eqref{vovermuim}, the proof is complete. 
\end{proof}
\begin{proof}[\textbf{Proof of Theorem \ref{qEXP}:}] \ \\
As in the proof of Theorem \ref{qCLT} this follows from Proposition \ref{generalprop}.
\end{proof}
\end{subsection}
\end{section}

\begin{section}{Stable Behavior of the Quenched Variance}\label{qvs}
Recall from Theorem \ref{VETtail} that $Q\left( Var_\w T_\nu > x \right) \sim K_\infty x^{-s/2}$. Since the sequence of random variables $\left\{ Var_\w (T_{\nu_i} - T_{\nu_{i-1}}) \right\}_{i\in \N}$ is stationary under $Q$ (and weakly dependent) it is somewhat natural to expect that $n^{-2/s} Var_\w T_{\nu_n}$ converges in distribution (under $Q$) to a stable law of index $\frac{s}{2}<1$.  

\begin{proof}[\textbf{Proof of Theorem \ref{Varstable}:}] \ \\
Obviously it is enough to prove that the second equality in \eqref{stableET2} holds and that
\begin{equation}
\lim_{n\ra\infty} Q\left( \left| Var_\w T_{\nu_n} - \sum_{i=1}^n (E_\w^{\nu_{i-1}} T_{\nu_i})^2 \right| > \d n^{2/s} \right) = 0, \quad \forall \d>0. \label{VET2diff} 
\end{equation}
However, \eqref{VET2diff} is the statement of \cite[Corollary 5.6]{pzSL1} with $m=\infty$.
Thus it is enough to prove the second equality in \eqref{stableET2}. To this end, first note that 
\begin{align}
\frac{1}{n^{2/s}}\sum_{i=1}^n \left( E_\w^{\nu_{i-1}} T_{\nu_i} \right)^2 &= 
\frac{1}{n^{2/s}} \sum_{i=1}^n \left( \left( E_\w^{\nu_{i-1}} T_{\nu_i} \right)^2 - \left( E_\w^{\nu_{i-1}} \bar{T}^{(n)}_{\nu_i} \right)^2 \right) \label{switchET2} \\
&\qquad + \frac{1}{n^{2/s}} \sum_{i=1}^n \left( E_\w^{\nu_{i-1}} \bar{T}^{(n)}_{\nu_i} \right)^2 \mathbf{1}_{M_i \leq n^{(1-\e)/s}}  \label{smallET2} \\
&\qquad + \frac{1}{n^{2/s}} \sum_{i=1}^n \left( E_\w^{\nu_{i-1}} \bar{T}^{(n)}_{\nu_i} \right)^2 \mathbf{1}_{M_i > n^{(1-\e)/s}}. \label{bigET2}
\end{align}
Therefore, it is enough to show that \eqref{switchET2} and \eqref{smallET2} converge to $0$ in distribution (under $Q$) and that
\begin{equation}
\lim_{n\ra\infty} Q\left( \frac{1}{n^{2/s}} \sum_{i=1}^n \left( E_\w^{\nu_{i-1}} \bar{T}^{(n)}_{\nu_i} \right)^2 \mathbf{1}_{M_i > n^{(1-\e)/s}} \leq x \right) = L_{\frac{s}{2},b}(x) \label{bigET2stable}
\end{equation}
for some $b>0$. To prove that \eqref{switchET2} converges to $0$ in distribution, first note that factoring gives
\[
 \left( E_\w^{\nu_{i-1}} T_{\nu_i} \right)^2 - \left( E_\w^{\nu_{i-1}} \bar{T}^{(n)}_{\nu_i} \right)^2  \leq 2 E_\w^{\nu_{i-1}} T_{\nu_i} \left(  E_\w^{\nu_{i-1}} T_{\nu_i} - E_\w^{\nu_{i-1}} \bar{T}^{(n)}_{\nu_i} \right).
\]
Therefore, for any $\d>0$
\begin{align}
&Q\left( \sum_{i=1}^n \left( \left( E_\w^{\nu_{i-1}} T_{\nu_i} \right)^2 - \left( E_\w^{\nu_{i-1}} \bar{T}^{(n)}_{\nu_i} \right)^2 \right) > \d n^{2/s} \right) \nonumber \\
&\qquad \leq  Q\left( \sum_{i=1}^n 2 E_\w^{\nu_{i-1}} T_{\nu_i} \left(  E_\w^{\nu_{i-1}} T_{\nu_i} - E_\w^{\nu_{i-1}} \bar{T}^{(n)}_{\nu_i}  \right) > \d n^{2/s} \right) \nonumber \\
&\qquad \leq n Q\left(  E_\w T_{\nu} - E_\w \bar{T}^{(n)}_{\nu}  > 1 \right) + Q\left( 2 E_\w T_{\nu_n} > \d n^{2/s} \right). \label{switchET2b}
\end{align}
Then, \cite[Lemma 3.2 \& Theorem 1.1]{pzSL1} give that both terms in \eqref{switchET2b} tend to zero as $n\ra\infty$. The proof that \eqref{smallET2} converges in distribution to $0$ is essentially a counting argument. Since the $M_i$ are all independent and from \eqref{Mtail} we know the asymptotics of $Q(M_i > x)$, we can get good bounds on the number of $i\leq n$ with $M_i\in(n^\a, n^\b]$. Then, since by \cite[(15)]{pzSL1} we have $Q\left( E_\w^{\nu_{i-1}} \bar{T}^{(n)}_{\nu_i} \geq n^\b, M_i\leq n^\a\right) = o\left( e^{-n^{(\b-\a)/5}} \right)$ we can also get good bounds on the number of $i\leq n$ with $E_\w^{\nu_{i-1}} \bar{T}^{(n)}_{\nu_i} \in(n^\a, n^\b]$. The details of this argument are essentially the same as the proof of Lemma 5.5 in \cite{pzSL1} and will thus be ommitted. 
Finally, we will use \cite[Theorem 5.1(III)]{kGPD} to prove \eqref{bigET2stable}. 
Now, Theorem \ref{VETtail} gives that $Q\left( \left( E_\w T_{\nu} \right)^2 \mathbf{1}_{M_1 > n^{(1-\e)/s}} > x n^{2/s} \right) \sim K_\infty x^{-s/2} n^{-1}$, and \cite[Lemma 3.4]{pzSL1} gives bounds on the mixing of the array $\left\{ \left( E_\w^{\nu_{i-1}} T_{\nu_i} \right)^2 \mathbf{1}_{M_i > n^{(1-\e)/s}} \right\}_{i\in\Z, n\in\N}$.
This is enough to verify the first two conditions of \cite[Theorem 5.1(III)]{kGPD}. The final condition that needs to be verified is
\begin{equation}
\lim_{\d\ra 0}\limsup_{n\ra\infty} n E_Q \left[
n^{-2/s} (E_\w \bar{T}_{\nu}^{(n)})^2 \mathbf{1}_{M_1>n^{(1-\e)/s} }
\mathbf{1}_{n^{-1/s} E_\w \bar{T}_{\nu}^{(n)}  \leq \d} \right] = 0 \,.
\label{truncexp}
\end{equation}
By Theorem \ref{VETtail} we have that 
there exists a constant $C_4>0$ such that
for any $x > 0$,
\[
Q\left( E_\w \bar{T}_{\nu}^{(n)} > x
n^{1/s} , M_1
> n^{(1-\e)/s} \right) \leq Q\left( E_\w T_\nu > x
n^{1/s} \right) \leq C_4 x^{-s}\frac{1}{n}.
\]
Then using this we have
\begin{align*}
& n E_Q \left[ n^{-2/s} \left( E_\w \bar{T}_{\nu}^{(n)} \right)^2 \mathbf{1}_{M_1>n^{(1-\e)/s} }
\mathbf{1}_{n^{-1/s}E_\w \bar{T}_{\nu}^{(n)}  \leq \d} \right] \\
&\quad = n \int_0^{\d^2} Q\left( \left(E_\w \bar{T}_{\nu}^{(n)}\right)^2 > x n^{2/s} ,
M_1 > n^{(1-\e)/s} \right)  dx \\
&\quad \leq C_4 \int_0^{\d^2} x^{-s/2} dx = \frac{C_4 \d^{2-s}}{1-s/2}\,,
\end{align*}
where the last
integral is finite since $s<2$. \eqref{truncexp} follows, and therefore by \cite[Theorem 5.1(III)]{kGPD} we have that \eqref{bigET2stable} holds.
\end{proof}
\end{section}

\textbf{Acknowledgments.} I would like to thank Olivier Zindy for his helpful comments regarding the analysis of the quenched Laplace transform of $\bar{T}^{(n)}_\nu$ in Section \ref{Laplace}.



\begin{thebibliography}{99}
\bibitem{ctMZ} 
	Y. S. Chow and H. Teicher,
	\emph{Probability theory: independence, interchangeability, martingales},
	Springer-Verlag, New York (1978).
\bibitem{ESZ} N. Enriquez, C. Sabot and O. Zindy, 
	Limit laws for transient random walks in random environment on $\mathbb Z$,
	{\it preprint} (2007), arXiv:math/0703660v1 [math.PR]
\bibitem{gsMVSS}
  N. Gantert and Z. Shi,
  Many Visits to a Single Site by a Transient Random Walk in Random Environment,
  \emph{Stochastic Process. Appl.} \textbf{99} (2002), \emph{no. 2}, pp. 159-176.
\bibitem{gQCLT}
  I. Y. Goldsheid,
  Simple Transient Random Walks in One-dimensional Random Environment: the Central Limit Theorem,
  \emph{math.PR/0605775}, (2006).
\bibitem{iEV}
  D. L. Iglehart,
  Extreme Values in the GI/G/1 Queue,
  \emph{Ann Math. Statist.} \textbf{43} (1972), pp. 627-635.
\bibitem{kksStable}
  H. Kesten, M. V. Kozlov, and F. Spitzer,
  A limit law for random walk in a random environment,
  \emph{Comp. Math} \textbf{30} (1975), pp. 145-168.
\bibitem{kGPD}
  M. Kobus,
  Generalized Poisson Distributions as Limits of Sums for Arrays of Dependent Random Vectors,
  \emph{J. Multivariate Anal.} \textbf{52} (1995), pp. 199-244
\bibitem{kmCLT}
  S. M. Kozlov and S. A. Molchanov,
  Conditions for the applicability of the central limit theorem to random walks on a lattice (Russian),
  \emph{Dokl. Akad. Nauk SSSR} \textbf{278} (1984), \emph{no. 3}, pp. 531-534.
\bibitem{pzSL1}
  J. Peterson and O. Zeitouni,
  Quenched Limits for Transient, Zero-Speed One-Dimensional Random Walk in Random Environment, to appear in \emph{Annals Probab.} (2008). 
\bibitem{pThesis}
  J. Peterson,
  \emph{PhD Thesis (Forthcoming, 2008)}.
\bibitem{sRWRE}
    F. Solomon,
    Random walks in random environments,
    \emph{Annals Probab.} \textbf{3} (1975), pp. 1-31.
\bibitem{zRWRE}
  O. Zeitouni,
  Random Walks in Random Environment in
  \emph{Lecture Notes in Mathematics} \textbf{1837},
  Springer, Berlin (2004).
\end{thebibliography}
\end{document}